\documentclass[12pt]{article}

\usepackage{amsmath}
\numberwithin{equation}{section}
\usepackage{amsfonts}
\usepackage{amssymb}
\usepackage{enumitem}
\usepackage[colorlinks=true,linkcolor=black,citecolor=blue,urlcolor=blue,]{hyperref}
\newtheorem{theorem}{Theorem}[section]
\newtheorem{lemma}{Lemma}[section]
\newtheorem{corollary}{Corollary}[section]

\newenvironment{proof}{\noindent{\bf Proof:}}{\hfill\fbox{}\vspace*{1mm}}

\begin{document}
	
	\title{Boundedness of Multiparameter Forelli-Rudin Type Operators on Product $L^p$ Spaces over Tubular Domains}

	\author{ Lvchang Li, Yuheng Liang, Haichou Li*}
	\author{Lvchang Li \thanks
		{College of Mathematics and Informatics,
			South China Agricultural University,
			Guangzhou,
			510640,
			China
			Email: 20222115006@stu.scau.edu.cn.},\
		Yuheng Liang\thanks{College of Mathematics and Informatics,
			South China Agricultural University,
			Guangzhou,
			510640,
			China
			Email: 20232115004@stu.scau.edu.cn.},\
		Haichou Li* \thanks{ Corresponding author,
			College of Mathematics and Informatics,
			South China Agricultural University,
			Guangzhou,
			510640,
			China
			Email: hcl2016@scau.edu.cn.  Li is supported by NSF of China (Grant No. 11901205). }
		\thanks{This work was supported by the NSF of China (Grant No. 12326407, and 12071155)}
	}
	
	\date{}
	\maketitle
	\begin{center}
		\begin{minipage}{120mm}
			\begin{center}{\bf Abstract}\end{center}
			{In this paper, we introduce and study two classes of multiparameter Forelli-Rudin type operators from $L^{\vec{p}}\left(T_B\times T_B, dV_{\alpha_1}\times dV_{\alpha_2}\right)$ to $L^{\vec{q}}\left(T_B\times T_B, dV_{\beta_1}\times dV_{\beta_2}\right)$, especially on their boundedness, where $L^{\vec{p}}\left(T_B\times T_B, dV_{\alpha_1}\times dV_{\alpha_2}\right)$ and $L^{\vec{q}}\left(T_B\times T_B, dV_{\beta_1}\times dV_{\beta_2}\right)$ are both weighted Lebesgue spaces over the Cartesian product of two tubular domains $T_B\times T_B$, with mixed-norm and appropriate weights. We completely characterize the boundedness of these two operators when $1\le \vec{p}\le \vec{q}<\infty$. Moreover, we provide the necessary and sufficient condition of the case that $\vec{q}=(\infty,\infty)$. As an application, we obtain the boundedness of three common classes of integral operators, including the weighted multiparameter Bergman-type projection and the weighted multiparameter Berezin-type transform.}		
			
		\end{minipage}
	\end{center}
	
	\maketitle

\section{Introduction}
\ \ \ \
The study of the analytic properties of  integral operators has always been a very active problem. In recent years, many scholars have studied its boundedness, compactness and other analytical properties, and obtained a series of important results, which can be referred to \cite{Book B C}. The Forelli-Rudin operator was first introduced by Forelli and Rudin \cite{F R} in 1974 as a special class of integral operators. They play an important role in the theory of function spaces, especially in the study of Bergman spaces and Hardy spaces. Initially, this research mainly focused on the operator theory of Bergman space and Hardy space on the unit disk and the unit ball. Subsequently, it is gradually extended to a wider range of function spaces, such as weighted Bergman space and Sobolev space. Researchers have found that the behavior of the Forelli-Rudin operator in these spaces can help to understand more general function space theory.

Kurens and Zhu \cite{F R X} introduced the Forelli-Rudin type integral operators in 2006, which are defined as follows. For any given $a, b, c\in \mathbb{R}$,
$$
T_{a,\ b,\ c}\ f\left( z \right) =\left( 1-\left| z \right|^2 \right) ^a\int_{\mathbb{B}_n}{\frac{\left( 1-\left| u \right|^2 \right) ^b}{\left( 1-\left< z,u \right> \right) ^c}f\left( u \right) dV\left( u \right)}
$$
and
$$
S_{a,\ b,\ c}\ f\left( z \right) =\left( 1-\left| z \right|^2 \right) ^a\int_{\mathbb{B}_n}{\frac{\left( 1-\left| u \right|^2 \right) ^b}{\left| 1-\left< z,u \right> \right|^c}f\left( u \right) dV\left( u \right)}.
$$

The above operators are defined on the unit ball. Stein \cite{Stein 1} first proved that the operator $T_{0,\ 0,\ n+1}$ is bounded on $L^p\left( \mathbb{B}_n \right)$, where $ 1 < p <\infty$. Forelli-Rudin \cite{F R} proved that $T_{0, \sigma +it,n+1+\sigma +it}$ is bounded on $L^p\left( \mathbb{B}_n \right)$ if and only if $(\sigma+1)p>1$, where $ 1 < p < \infty, \sigma > -1, t \in \mathbb{R}$, $i$ is an imaginary unit. Zhu \cite[Theorem 3.11]{zhu2007operator} gave the necessary and sufficient conditions for $T_{a,\ b,\ c}$ and $S_{a,\ b,\ c}$ to be bounded on $L^p\left( \mathbb{D},dA_{\alpha} \right) $, where $1 < p < \infty,\ dA_{\alpha}(z) = (\alpha+1) \left (1-\left| z \right|^2 \right)^{\alpha}dV(z), \alpha > -1$. Later, this conclusion was extended to the high-dimensional case by Kures and Zhu \cite{F R X}. For more results in this direction, please refer to \cite{F R jin1,F R jin2,Zhao 1,Zhao 2}. In addition, it also includes the study of Forelli-Rudin operators on special domains (such as unbounded domains), revealing the behavior of operators in different geometric backgrounds. For example, Cheng et al. \cite{Cheng} gave the boundedness of $T_{0,\ 0,\ c}$ from $L^p$ to $L^ q$ when $(p, q)\in [1,\infty]\times[1,\infty]$. In \cite{Liu 2}, Liu et al. generalized the results of Cheng et al.to the Siegel upper half space, and discussed the $L^p$-$L^q$ boundedness of the operator $T_{0,\ 0,\ c}$. In \cite{Zhou 1}, Zhou et al. gave the necessary and sufficient conditions for the $L_{\alpha}^p$-$L_{\beta}^q$ boundedness of Forelli-Rudin type integral operators on Siegel upper half space, and gave the connection between the unit ball and the boundedness of Forelli-Rudin type integral operators on Siegel upper half space. For more results in the Siegel upper half space, please refer to \cite{Liu 1,Wang}.

The mixed norm Lebesgue space $L^{\vec{p}}$ as a natural generalization of the classical Lebesgue space $L^p$ was first introduced by Benedek \cite{mixed norm} in 1961. Its definition and properties combine multiple Lebesgue integrals in different directions, which can better describe the behavior of multidimensional functions. Mixed norm Lebesgue spaces are increasingly showing their importance in many analysis and application fields, especially in multivariate function analysis and partial differential equation research. By combining Lebesgue integrals in different directions, the mixed norm Lebesgue space can better describe and deal with the behaviors of complex functions. After that, many scholars have studied it. The research on this aspect can refer to the relevant literature \cite{mixed norm 1}. There is also progress in the study of Forelli-Rudin type operators on Lebesgue spaces with mixed norm. Recently, Huang et al. \cite{M Z} studied the boundedness of multi-parameter Forelli-Rudin type operators on weighted Lebesgue spaces $L_{\vec{\alpha}}^{\vec{p}}\left( \mathbb{B}_n\times \mathbb{B}_n \right) $ with mixed norm on the Cartesian product of two unit balls, and gave a series of conclusions. The object of this paper is the boundedness of multi-parameter Forelli-Rudin type operators on the weighted Lebesgue space $L_{\vec{\alpha}}^{\vec{p}}\left( T_B\times T_B \right) $ with mixed norm on the Cartesian product of two tubular domains. It also includes the boundedness of multi-parameter Forelli-Rudin type operators on $L_{\vec{\alpha}}^{\vec{p}}\left( T_B\times T_B \right) $ when $\vec{q}:=(\infty,\infty)$. The detailed conclusions are shown in section 6 of this paper.

In addition, multiparameter theory is also an important branch of mathematics. It provides a basis for dealing with complex mathematical problems involving multiple variables or parameters, and plays an important role in harmonic analysis, complex analysis, partial differential equations and other fields. For example, the singular integral operator studied by many scholars in harmonic analysis, the product singular integral studied by R.Fefferman, Journé and Stein et al. \cite{R. Fefferman,J. Journé}. Specifically, in 1982, R.Fefferman and Stein \cite{Stein 1} extended the convolution Calderón–Zygmund operator to the two-parameter case and obtained the boundedness on the product Lebesgue space. Gundy and Stein \cite{Gundy} first introduced the product Hardy space $H^p\left(\mathbb{R}^n\times\mathbb{R}^m\right)$. Subsequently, Chang and R.Fefferman \cite{R. Fefferman 2,R. Fefferman 3} established the Calderón–Zygmund decomposition theory of the product space and derived a series of conclusions. For more applications of multi-parameters in the above fields, see reference \cite{Han Y S,Stein 3,Stein 4,Stein 2}.

The rest of this article is organized as follows. In the second section, we will explain the basic concepts and symbolic terms. In the third section, we give some basic lemmas, including Schur test, which is an important tool to verify the boundedness of integral operators and will run through the full text. In the fourth section, we will give the necessary conditions for the boundedness of the operator $T_{\vec{a},\ \vec{b},\ \vec{c}}$, and get a series of conclusions, see lemma \ref {lem b1} to lemma \ref{lem b13}. In the fifth section, we give the sufficient conditions for the boundedness of the operator $S_{\vec{a},\ \vec{b},\ \vec{c}}$, see Lemma \ref{lem 1 1} to Lemma \ref{lem 1 13}. Since the boundedness of the operator $S_{\vec{a},\ \vec{b},\ \vec{c}}$ will inevitably lead to the boundedness of the operator $T_{\vec{a},\ \vec{b},\ \vec{c}}$, in the sixth section, we link the above lemmas and give our main theorems \ref{th 1} to \ref{th 13} in this paper. In the last section, we apply the obtained theorem to three kinds of integral operators, including the famous Bergman-type projection and Berezin-type transformation, and obtain the boundedness of three kinds of integral operators.

\section{Preliminaries}
\ \ \ \
Let $\mathbb{C}^n$ be the $n$-dimensional complex Euclidean space. For any two points $z=\left( z_1,\cdots ,z_n \right) $ and $w=\left( w_1,\cdots ,w_n \right) $ in $\mathbb{C}^n,$ we write\[z\cdot \bar{w}:=z_1\bar{w}_1+\cdots +z_n\bar{w}_n,\] \[z'^2=z'\cdot z'=z_{1}^{2}+z_{2}^{2}+\cdots +z_{n}^{2}\] and \[\left| z \right|:=\sqrt{z\cdot \bar{z}}=\sqrt{\left| z_1 \right|^2+\cdots +\left| z_n \right|^2}.\]  

The set $T_B=\left\{ z=x+iy,x\in \mathbb{R}^n,y\in B \right\}$ is a tube domain in an n-dimensional complex space $\mathbb{C}^n$, where\[B=\left\{ \left( y',y_n \right) \in \mathbb{C}^n\left| y'^2<y_n \right. \right\},y'=\left( y_1,y_2,\cdots ,y_{n-1} \right) \in \mathbb{C}^{n-1}, y_n \in \mathbb{C}.\] 

For any given $\vec{p}:=\left( p_1,p_2 \right) \in \left( 0,\infty \right) \times \left( 0,\infty \right) $ and $\vec{\alpha}:=\left( \alpha _1,\alpha _2 \right) \in \left( -1,\infty \right) \times \left( -1,\infty \right)$,
we define the Lebesgue space $L_{\vec{\alpha}}^{\vec{p}}\left( T_B\times T_B \right)$ with weighted mixed-norm on $T_B\times T_B$. It consists of all Lebesgue measurable functions $f$ on $T_B\times T_B$ such that the norm
$$
\lVert f \rVert _{\vec{p},\vec{\alpha}}=\left\{ \int_{T_B}{\left( \int_{T_B}{\left| f\left( z,w \right) \right|^{p_1}dV_{\alpha _1}\left( z \right)} \right) ^{\frac{p_2}{p_1}}dV_{\alpha _2}\left( w \right)} \right\} ^{\frac{1}{p_2}}
$$
is finite, where for any $i\in\{1,2\}$, $dV_{\alpha _i}\left( z \right) =\rho(z)^{\alpha _i}dV\left( z \right)$, with $\rho(z):= y_n-\left| y^{\prime} \right|^2$ and $dV$ denoting the Lebesgue volume measure on $\mathbb{C}^n$. 

In particular, when $p_1=p_2=p$ and $\alpha_1=\alpha_2=\alpha$, the space $L_{\vec{\alpha}}^{\vec{p}}\left( T_B\times T_B \right) $ goes back to the weighted Lebesgue space $L_{\alpha}^p\left( T_B\times T_B \right) $.

Similarly, we define the space of all essentially bounded functions on $T_B\times T_B$, denoted as $L^{\infty}\left( T_B\times T_B \right)$.

Same as the well-known classical $L^p$ space, the weighted mixed-norm Lebesgue space $L_{\vec{\alpha}}^{\vec{p}}\left( T_B\times T_B \right) $ that we define is a Banach space under the norm $\lVert \cdot \rVert _{\vec{p},\vec{\alpha}}$ when $\vec{p}=\left(p_1,p_2\right)\in \left[ 1,\infty \right) \times \left[ 1,\infty \right) $.
\\

Next, we introduce two important integral operators $T_{\vec{a},\ \vec{b},\ \vec{c}}$ and $S_{\vec{a},\ \vec{b},\ \vec{c}}$ studied in this paper.

For any $\vec{a}:=\left( a_1,a_2 \right) ,\vec{b}:=\left( b_1,b_2 \right) ,\vec{c}:=\left( c_1,c_2 \right) \in \mathbb{R}^2$, two classes of integral operators are defined by 
$$
T_{\vec{a},\ \vec{b},\ \vec{c}}\ f\left( z,w \right) =\rho \left( z \right) ^{a_1}\rho \left( w \right) ^{a_2}\int_{T_B}{\int_{T_B}{\frac{\rho \left( u \right) ^{b_1}\rho \left( \eta \right) ^{b_2}}{\rho \left( z,u \right) ^{c_1}\rho \left( w,\eta \right) ^{c_2}}f\left( u,\eta \right) d V\left( u \right)}d V\left( \eta \right)}
$$
and
$$
S_{\vec{a},\ \vec{b},\ \vec{c}}\ f\left( z,w \right) =\rho \left( z \right) ^{a_1}\rho \left( w \right) ^{a_2}\int_{T_B}{\int_{T_B}{\frac{\rho \left( u \right) ^{b_1}\rho \left( \eta \right) ^{b_2}}{\left| \rho \left( z,u \right) \right|^{c_1}\left| \rho \left( w,\eta \right) \right|^{c_2}}f\left( u,\eta \right) dV\left( u \right)}dV\left( \eta \right)},
$$
where  
$$
\rho \left( z,u \right): =\frac{1}{4}\left( \left( z'-\overline{u'} \right) ^2-2i\left( z_n-\overline{u_n} \right) \right).
$$

In this paper, we stipulate that the mixed norm Lebesgue space $L^{\vec{p}}$ is denoted by $L^\infty$ when only two indexes are infinite, and the rest are denoted by $L^{\vec{p}}$. 

For the convenience of writing, we write $p_-:=\min \left\{ p_1,p_2 \right\}$, $p_+:=max\left\{p_1,p_2\right\}$, $q_-:=min\left\{q_1,q_2\right\}$ and $q_+:=max\left\{q_1,q_2\right\}$.

Throughout the paper we use $C$ to denote positive constants whose value may change from line to line but does not depend on the functions being considered. The notation $A\lesssim B$ means that there is a positive constant $C$ such that $A\le CB$, and the notation $A\simeq B$ means that $A\lesssim B$ and $B\lesssim A$.

\section{Basic lemmas}
\ \ \ \ 
In this section, we will introduce several key lemmas, which will play an important role in the proof of the main theorems.

The following lemmas \ref{lem jifendengshi} and \ref{lem qu bian liang} are derived from \cite{jiaxin2023bergman} and \cite{Li} respectively, which play an important role in integral estimation.

\begin{lemma}\label{lem jifendengshi}
	Let $r, s>0, t>-1$ and $r+s-t>n+1$, then 
	\[  \int_{T_{B}}\frac{{\rho (w)}^{t}}{{\rho (z,w)}^{r}{\rho (w,u)}^{s}}dV(w)=\frac{C_{1}(n,r,s,t)}{\rho (z,u)^{r+s-t-n-1}}        \]
	for all $z,u \in T_{B}$, where
	$$C_{1}(n,r,s,t)=\frac{2^{n+1}{\pi}^{n}\Gamma (1+t)\Gamma (r+s-t-n-1)}{\Gamma (r)\Gamma (s)}.$$
	In particular, let $s,t\in \mathbb{R}$, if $t>-1,s-t>n+1$, then \[\int_{T_B}{\frac{\rho \left( w \right) ^t}{\left| \rho \left( z,w \right) \right|^s}}dV\left( w \right) =\frac{C_1\left( n,s,t \right)}{\rho \left( z \right) ^{s-t-n-1}}.\] Otherwise, the above equation is infinity.
\end{lemma}

\begin{lemma}\label{lem qu bian liang}
	For any $z,w\in {T_{B}}$, we have 
	$$
	2\left| \rho \left( z,w \right) \right|\ge \max \left\{ \rho \left( z \right) ,\rho \left( w \right) \right\}.
	$$
\end{lemma}

The following lemma \ref{dui ou} is an important result in real analysis. Detailed proof can refer to \cite{Real}.

\begin{lemma}\label{dui ou}
	Suppose $(X,d\mu)$ is a $\sigma$-finite measure space, $1\le p<\infty$, and $1/p+1/q=1$. Let $G$ be a complex-valued function defined on $X\times X$ and $T$ be the integral operator defined by
	$$
	Tf\left( x \right) =\int_X{G\left( x,y \right) f\left( y \right) d\mu \left( y \right) .}
	$$
	If the operator $T$ is bounded on $L^p(X,d\mu)$, then its adjoint $T^*$ is the integral oparator 
	$$
	T^*f\left( x \right) =\int_X{\overline{G\left( y,x \right) }f\left( y \right) d\mu \left( y \right)}
	$$
	on $L^q(X,d\mu)$.
\end{lemma}

The following lemma \ref{dui ou you jie} is a further conclusion on the basis of lemma \ref{dui ou}.
\begin{lemma}\label{dui ou you jie}
	Let $\vec{p}:=\left( p_1,p_2 \right) \in \left[ 1,\infty \right]\times\left[ 1,\infty \right]$ and $\vec{q}:=\left( q_1,q_2 \right) \in \left[ 1,\infty \right]\times\left[ 1,\infty \right]$. If the integral operator $T_{\vec{a},\ \vec{b},\ \vec{c}}$ is bounded from $L_{\vec{\alpha}}^{\vec{p}}\left( T_B\times T_B \right) $ to $L_{\vec{\beta}}^{\vec{q}}\left( T_B\times T_B \right) $, then its adjoint operator $T_{\vec{a},\ \vec{b},\ \vec{c}}^*$ defined by setting 
	$$
	T_{\vec{a},\ \vec{b},\ \vec{c}}^{*}\ f\left( z,w \right) =\rho \left( z \right) ^{b_1-\alpha _1}\rho \left( w \right) ^{b_2-\alpha _2}\int_{T_B}{\int_{T_B}{\frac{\rho \left( u \right) ^{\beta _1+a_1}\rho \left( \eta \right) ^{\beta _2+a_2}}{\rho \left( z,u \right) ^{c_1}\rho \left( w,\eta \right) ^{c_2}}f\left( u,\eta \right) dV\left( u \right)}dV\left( \eta \right)}
	$$
	is bounded from $L_{\vec{\beta}}^{\vec{q}^{\prime}}\left( T_B\times T_B \right)$ to $L_{\vec{\alpha}}^{\vec{p}^{\prime}}\left( T_B\times T_B \right)$.
\end{lemma}

\begin{proof}
	It is easy to get the definition of operator $T_{\vec{a},\ \vec{b},\ \vec{c}}^{*}$ from Lemma \ref{dui ou}.

	Let $g(z,w)\in L_{\vec{\alpha}}^{\vec{q}'}\left( T_B\times T_B \right)$,  it follows from the Fubini theorem and the Hölder inequality of the mixed norm that
	$$
	\begin{aligned}
	\lVert T_{\vec{a},\ \vec{b},\ \vec{c}}^{*}\ g (z,w)\rVert _{\vec{p}^{\prime},\vec{\alpha}}=&\mathop{\text{sup}}_{\lVert f \rVert _{\vec{p},\vec{\alpha}}=1}\left| \int_{T_B}{\int_{T_B}{T_{\vec{a},\ \vec{b},\ \vec{c}}^{*}\ g\left( z,w \right) \overline{f\left( z,w \right) }dV_{\alpha _1}\left( z \right)}dV_{\alpha _2}\left( w \right)} \right|\\	
	=&\mathop{\text{sup}}_{\lVert f \rVert _{\vec{p},\vec{\alpha}}=1}\left| \int_{T_B}\int_{T_B}\int_{T_B}\int_{T_B}\frac{\rho \left( u \right) ^{\beta _1+a_1}\rho \left( \eta \right) ^{\beta _2+a_2}}{\rho \left( z,u \right) ^{c_1}\rho \left( w,\eta \right) ^{c_2}} \right.\\
	&\left.\times g\left( u,\eta \right) dV\left( u \right)dV\left( \eta \right)\overline{f\left( z,w \right) }dV_{\alpha _1}\left( z \right)dV_{\alpha _2}\left( w \right) \right|\\
	=&\mathop{\text{sup}}_{\lVert f \rVert _{\vec{p},\vec{\alpha}}=1}\left| \int_{T_B}{\int_{T_B}{g\left( u,\eta \right) \overline{T_{\vec{a},\ \vec{b},\ \vec{c}}\ f\left( u,\eta \right) }dV_{\beta _1}\left( u \right)}dV_{\beta _2}\left( \eta \right)} \right|\\
	\le &\mathop{\text{sup}}_{\lVert f \rVert _{\vec{p},\vec{\alpha}}=1}\lVert T_{\vec{a},\ \vec{b},\ \vec{c}}\ f \rVert _{\vec{q},\vec{\beta}}\ \lVert g \rVert _{\vec{q}',\vec{\beta}}\\
	\lesssim &\lVert g \rVert _{\vec{q}',\vec{\beta}}	.
	\end{aligned}
	$$
		
	Therefore, the operator $T_{\vec{a},\ \vec{b},\ \vec{c}}^{*}$ is a bounded operator from $L_{\vec{\beta}}^{\vec{q}^{\prime}}\left( T_B\times T_B \right)$ to $L_{\vec{\alpha}}^{\vec{p}^{\prime}}\left( T_B\times T_B \right)$.	
\end{proof}

In order to prove the conclusion of the fifth section of this article, we introduce the following four lemmas about the Schur's test. All lemmas are derived from \cite{M Z}. Therefore,  we omit their proofs.

\begin{lemma}\label{S1}
	Let $\vec{\mu}:=\mu_1 \times \mu_2$ and $\vec{v}:=v_1 \times v_2$ be positive measures on the space $X \times X$ and, for $i \in\{1,2\}, K_i$ be nonnegative functions on $X \times X$. Let $T$ be an integral operator with kernel $K:=K_1 \cdot K_2$ defined by setting for any $(x, y) \in X \times X$,
	$$
	T f(x, y):=\int_X \int_X K_1(x, s) K_2(y, t) f(s, t) d \mu_1(s) d \mu_2(t) .
	$$
	Suppose $\vec{p}:=\left(p_1, p_2\right)\in (1,\infty)\times(1,\infty), \vec{q}:=\left(q_1, q_2\right) \in(1, \infty)\times(1,\infty)$ satisfying $1<p_{-} \leq p_{+} \leq q_{-}<\infty$, where $p_{+}:=\max \left\{p_1, p_2\right\}, p_{-}:=\min \left\{p_1, p_2\right\}$, and $q_{-}:=\min \left\{q_1, q_2\right\}$. Let $\gamma_i$ and $\delta_i$ be real numbers such that $\gamma_i+\delta_i=1$ for $i \in\{1,2\}$. If there exist two positive functions $h_1$ and $h_2$ defined on $X \times X$ with two positive constants $C_1$ and $C_2$ such that for almost all $(x, y) \in X \times X$
	\begin{equation}\label{3.1}
	\int_X\left (\int_X\left (K_1(x, s)\right)^{\gamma_1 p_1^{\prime}}\left (K_2(y, t)\right)^{\gamma_2 p_1^{\prime}}\left (h_1(s, t)\right)^{p_1^{\prime}} d \mu_1(s)\right)^{p_2^{\prime} / p_1^{\prime}} d \mu_2(t) \leq C_1\left (h_2(x, y)\right)^{p_2^{\prime}}
	\end{equation}
	and for almost all $(s, t) \in X \times X$,
	\begin{equation}\label{3.2}
	\int_X\left (\int_X\left (K_1(x, s)\right)^{\delta_1 q_1}\left (K_2(y, t)\right)^{\delta_2 q_1}\left (h_2(x, y)\right)^{q_1} d \nu_1(x)\right)^{q_2 / q_1} d v_2(y) \leq C_2\left (h_1(s, t)\right)^{q_2},
	\end{equation}	
	then $T: L_{\vec{\mu}}^{\vec{p}} \rightarrow L_{\vec{v}}^{\vec{q}}$ is bounded with $\|T\|_{L_{\vec{\mu}}^{\vec{p}} \rightarrow L_{\vec{v}}^{\vec{q}}} \leq C_1^{1 / p_2^{\prime}} C_2^{1 / q_2}$.
\end{lemma}

\begin{lemma}\label{S2}
	Let $\vec{\mu}, \vec{v}$, the kernel $K$, and the operator $T$ be as in Lemma $\ref{S1}$. Suppose $\vec{q}:=\left(q_1, q_2\right) \in[1, \infty)\times[1, \infty)$. Let $\gamma_i$ and $\delta_i$ be two real numbers such that $\gamma_i+\delta_i=1$ for $i \in\{1,2\}$. If there exist two positive functions $h_1$ and $h_2$ defined on $X \times X$ with two positive constants $C_1$ and $C_2$ such that for almost all $(x, y) \in X \times X$,
	\begin{equation}\label{3.3}
	\underset{(s, t) \in X \times X}{\operatorname{ess} \sup }\left (K_1(x, s)\right)^{\gamma_1}\left (K_2(y, t)\right)^{\gamma_2} h_1(s, t) \leq C_1 h_2(x, y)
	\end{equation}	
	and, for almost all $(s, t) \in X \times X$,
	\begin{equation}\label{3.4}
	\int_X\left (\int_X\left (K_1(x, s)\right)^{\delta_1 q_1}\left (K_2(y, t)\right)^{\delta_2 q_1}\left (h_2(x, y)\right)^{q_1} d \nu_1(x)\right)^{q_2 / q_1} d \nu_2(y) \leq C_2\left (h_1(s, t)\right)^{q_2},
	\end{equation}
	then $T: L_{\vec{\mu}}^{\vec{1}} \rightarrow L_{\vec{v}}^{\vec{q}}$ is bounded with $\|T\|_{L_{\vec{\mu}}^{\vec{1}} \rightarrow L_{\vec{v}}^{\vec{q}}} \leq C_1 C_2^{1 / q_2}$.
\end{lemma}

\begin{lemma}\label{S3}
	Let $\vec{\mu}, \vec{v}$, the kernel $K$, and the operator $T$ be as in Lemma $\ref{S1}$. Suppose $\vec{p}=\left(p_1, 1\right)$ with $p_1 \in(1, \infty)$ and $\vec{q}:=\left(q_1, q_2\right) \in(1, \infty)\times(1, \infty)$ satisfying $1<p_1 \leq q_{-}<\infty$, where $q_{-}:=\min \left\{q_1, q_2\right\}$. Let $\gamma_i$ and $\delta_i$ be two real numbers such that $\gamma_i+\delta_i=1$ for $i \in\{1,2\}$. If there exist two positive functions $h_1$ and $h_2$ defined on $X \times X$ with two positive constants $C_1$ and $C_2$ such that for almost all $(x, y) \in X \times X$,
	\begin{equation}\label{3.5}
	\underset{t \in X}{\operatorname{ess} \sup } \int_X\left(K_1(x, s)\right)^{\gamma_1 p_1^{\prime}}\left(K_2(y, t)\right)^{\gamma_2 p_1^{\prime}}\left(h_1(s, t)\right)^{p_1^{\prime}} d \mu_1(s) \leq C_1\left (h_2(x, y)\right)^{p_1^{\prime}}
	\end{equation}
	and, for almost all $(s, t) \in X \times X$,
	\begin{equation}\label{3.6}
	\int_X\left (\int_X\left (K_1(x, s)\right)^{\delta_1 q_1}\left (K_2(y, t)\right)^{\delta_2 q_1}\left (h_2(x, y)\right)^{q_1} d \nu_1(x)\right)^{q_2 / q_1} d \nu_2(y) \leq C_2\left (h_1(s, t)\right)^{q_2},
	\end{equation}
	then $T: L_{\vec{\mu}}^{\vec{p}} \rightarrow L_{\vec{v}}^{\vec{q}}$ is bounded with $\|T\|_{L_{\vec{\mu}}^{\vec{p}} \rightarrow L_{\vec{v}}^{\vec{q}}} \leq C_1^{1 / p_1^{\prime}} C_2^{1 / q_2}$
\end{lemma}

\begin{lemma}\label{S4}
	Let $\vec{\mu}, \vec{v}$, the kernel $K$, and the operator $T$ be as in Lemma $\ref{S1}$. Suppose $\vec{p}=\left(1, p_2\right)$ with $p_2 \in(1, \infty)$ and $\vec{q}:=\left(q_1, q_2\right) \in(1, \infty)\times(1, \infty)$ satisfying $1<p_2 \leq q_{-}<\infty$, where $q_{-}:=\min \left\{q_1, q_2\right\}$. Let $\gamma_i$ and $\delta_i$ be two real numbers such that $\gamma_i+\delta_i=1$ for $i \in\{1,2\}$. If there exist two positive functions $h_1$ and $h_2$ defined on $X \times X$ with two positive constants $C_1$ and $C_2$ such that for almost all $(x, y) \in X \times X$,
	\begin{equation}\label{3.7}
	\int_X\left (\underset{s \in X}{\operatorname{ess} \sup }\left (K_1(x, s)\right)^{\gamma_1}\left (K_2(y, t)\right)^{\gamma_2} h_1(s, t)\right)^{p_2^{\prime}} d \mu_2(t) \leq C_1\left (h_2(x, y)\right)^{p_2^{\prime}}
	\end{equation}
	and, for almost all $(s, t) \in X \times X$,
	\begin{equation}\label{3.8}
	\int_X\left (\int_X\left (K_1(x, s)\right)^{\delta_1 q_1}\left (K_2(y, t)\right)^{\delta_2 q_1}\left (h_2(x, y)\right)^{q_1} d v_1(x)\right)^{q_2 / q_1} d v_2(y) \leq C_2\left (h_1(s, t)\right)^{q_2},
	\end{equation}
	then $T: L_{\vec{\mu}}^{\vec{p}} \rightarrow L_{\vec{v}}^{\vec{q}}$ is bounded with $\|T\|_{L_{\vec{\mu}}^{\vec{p}} \rightarrow L_{\vec{v}}^{\vec{q}}} \leq C_1^{1 / p_2^{\prime}} C_2^{1 / q_2}$.
\end{lemma}

Finally, we introduce a sufficient condition for the boundedness of an operator $S_{\vec{a},\ \vec{b},\ \vec{c}}$ when its codomain is $L^{\infty}\left( T_B\times T_B \right)$.

\begin{lemma}\label{ dao wu qiong }
	Let $\vec{\mu}:=\mu_1 \times \mu_2$ and $\vec{v}:=v_1 \times v_2$ be positive measures on the space $X \times X$ and, for $i \in\{1,2\}, K_i$ be nonnegative functions on $X \times X$. Let $T$ be an integral operator with kernel $K:=K_1 \cdot K_2$ defined by setting for any $(x, y) \in X \times X$,
	$$
	T f(x, y):=\int_X \int_X K_1(x, s) K_2(y, t) f(s, t) d \mu_1(s) d \mu_2(t) .
	$$
	Suppose $\vec{p}:=(p_1,p_2), \vec{q}:=(\infty, \infty)$ satisfying $1\le p_-\le p_+\le \infty$, if 
	\begin{equation}\label{3.9}
	\lVert K_1\left(x,\cdot \right) \cdot K_2\left(y,\cdot \right) \rVert _{L_{\vec{\mu}}^{\vec{p}^{\prime}}}
	\end{equation} 
	is uniformly bounded, then $T$ is bounded from $L_{\vec{\mu}}^{\vec{p}}$ to $L^{\infty}$. 
\end{lemma}

\begin{proof}
	This result can be easily obtained from the following mixed norm Hölder inequality, which is as follows : 
	$$
	\left| Tf(x,y) \right|\le \left\| K_1\left( x,\cdot \right) \cdot K_2\left( y,\cdot \right) \right\| _{L_{\vec{\mu}}^{\vec{p}\prime}}\left\| f \right\| _{L_{\vec{\mu}}^{\vec{p}}}.
	$$
\end{proof}

\section{The Necessity for the Boundedness of $T_{\vec{a},\ \vec{b},\ \vec{c}}$}
\ \ \ \
In this section, we give the necessary conditions for the boundedness of the operator $T_{\vec{a},\ \vec{b},\ \vec{c}}$, and the conclusions are the following lemmas \ref{lem b1} to lemma \ref{lem b13}.
\begin{lemma}\label{lem b1}
	Let $1<p_1\le q_1<\infty$ and $1<p_2\le q_2<\infty$. If the operator $T_{\vec{a},\ \vec{b},\ \vec{c}}$ is bounded from $L_{\vec{\alpha}}^{\vec{p}}\left( T_B\times T_B \right) $ to $L_{\vec{\beta}}^{\vec{q}}\left( T_B\times T_B \right) $, then for any $i \in \{1,2\}$, 
	$$
	\begin{cases}
	-q_ia_i<\beta _i+1,\  \alpha _i+1<p_i\left( b_i+1 \right) ,\\
	c_i=n+1+a_i+b_i+\lambda_i,\\
	\end{cases}
	$$
	\\
	where 
	$$
	\lambda_i =\frac{n+1+\beta _i}{q_i}-\frac{n+1+\alpha _i}{p_i}.
	$$
\end{lemma}

\begin{proof}
	By duality, the boundedness of $T_{\vec{a},\ \vec{b},\ \vec{c}}$ from $L_{\vec{\alpha}}^{\vec{p}}\left( T_B\times T_B \right) $ to $L_{\vec{\beta}}^{\vec{q}}\left( T_B\times T_B \right) $ implies the boundedness of $T_{\vec{a},\ \vec{b},\ \vec{c}}^{*}$ from $L_{\vec{\beta}}^{\vec{q}'}\left( T_B\times T_B \right) $ to $L_{\vec{\alpha}}^{\vec{p}'}\left( T_B\times T_B \right) $. 
	
	In order for $Tf$ to be well-defined, then for any $f\in L_{\vec{\alpha}}^{\vec{p}}\left( T_B\times T_B \right)$, the following formula must hold : 
	$$
	\frac{\rho(z)^{a_1}\rho \left( u \right) ^{b_1-\alpha _1}\rho(w)^{a_2}\rho \left( \eta \right) ^{b_2-\alpha _2}}{\rho \left( z,u \right) ^{c_1}\rho \left( w,\eta \right) ^{c_2}}\in L_{\vec{\alpha}}^{\vec{p}\prime}\left( T_B\times T_B \right),
	$$
	for any $z, \ w\in T_B$,

	that is 
	\begin{equation}\label{C1 t1}
	\rho(z)^{a_1}\left( \int_{T_B}{\frac{\rho \left( u \right) ^{p_{1}^{\prime}\left( b_1-\alpha _1 \right) +\alpha _1}}{\left| \rho \left( z,u \right) \right|^{c_1p_{1}^{\prime}}}dV\left( u \right)} \right) ^{\frac{1}{p_{1}^{\prime}}}
	\rho(w)^{a_2}\left( \int_{T_B}{\frac{\rho \left( \eta \right) ^{p_{2}^{\prime}\left( b_2-\alpha _2 \right) +\alpha _2}}{\left| \rho \left( w,\eta \right) \right|^{c_2p_{2}^{\prime}}}dV\left( \eta \right)} \right) ^{\frac{1}{p_{2}^{\prime}}}
	<\infty .	
	\end{equation}
	By Lemma \ref{lem jifendengshi}, (\ref{C1 t1}) is ture if and only if 
	$$
	\begin{cases}
	p_{i}^{\prime}\left( b_i-\alpha _i \right) +\alpha _i>-1, \\
	c_ip_{i}^{\prime}>n+1+p_{i}^{\prime}\left( b_i-\alpha _i \right) +\alpha _i\\
	\end{cases}
	$$
	for any $i\in \{1,2\}$.
	
	Summing up the two inequalities, we get $c_i>n/{p_i^{\prime}}>0$ for any $i\in \{1,2\}$.

	For $u>0$ and $\eta>0$, we put 
	$$
	f_{u,\eta}\left( z,w \right) =\frac{\rho \left( z \right) ^{t_1}\rho \left( w \right) ^{t_2}}{\rho \left( z,u\mathbf{i} \right) ^{s_1}\rho \left( w,\eta \mathbf{i} \right) ^{s_2}}, \ z,w\in T_B,
	$$
	where $s_1,s_2,t_1$ and $t_2$ are real parameters satisfying conditions
	\begin{equation}\label{b 11}
	\begin{cases}
	s_i>0,\\
	t_i>\max \left\{ -\frac{\beta _i+1}{q_{i}^{\prime}},-a_i-\beta _i-1 \right\} ,\\
	s_i-t_i>\max \left\{ \frac{n+1+\beta _i}{q_{i}^{\prime}},\beta _i+a_i-c_i+n+1 \right\} ,\\
	\end{cases}
	\end{equation}
	and $\mathbf{i}=\left(0^{\prime},i\right)$.
	
	By lemma \ref{lem jifendengshi}, condition \ref{b 11} makes $f_{u,\eta}\left( z,w \right)\in L_{\vec{\beta}}^{\vec{q}'}\left( T_B\times T_B \right)$, and there exists a constant $C$ such that 
	$$
	\lVert f_{u,\eta}\left( z,w \right) \rVert _{\vec{q}^{\prime},\vec{\beta}}=Cu^{\frac{n+1+\beta _1}{q_{1}^{\prime}}-\left( s_1-t_1 \right)}\eta ^{\frac{n+1+\beta _2}{q_{2}^{\prime}}-\left( s_2-t_2 \right)}.
	$$
	
	By Lemma \ref{lem jifendengshi}and \ref{dui ou you jie}, we have
	$$
	\begin{aligned}
	&T_{\vec{a},\ \vec{b},\ \vec{c}}^{*}\ f_{u,\eta}\left( z,w \right) \\
	&=\rho \left( z \right) ^{b_1-\alpha _1}\rho \left( w \right) ^{b_2-\alpha _2}\int_{T_B}{\int_{T_B}{\frac{\rho \left( l \right) ^{\beta _1+a_1}\rho \left( m \right) ^{\beta _2+a_2}}{\rho \left( z,l \right) ^{c_1}\rho \left( w,m \right) ^{c_2}}f_{u,\eta}\left( l,m \right) dV\left( l \right)}dV\left( m \right)}
	\\
	&=\rho \left( z \right) ^{b_1-\alpha _1}\rho \left( w \right) ^{b_2-\alpha _2}\int_{T_B}{\int_{T_B}{\frac{\rho \left( l \right) ^{\beta _1+a_1}\rho \left( m \right) ^{\beta _2+a_2}}{\rho \left( z,l \right) ^{c_1}\rho \left( w,m \right) ^{c_2}}\cdot \frac{\rho \left( l \right) ^{t_1}\rho \left( m \right) ^{t_2}}{\rho \left( l,u\mathbf{i} \right) ^{s_1}\rho \left( m,\eta \mathbf{i} \right) ^{s_2}}dV\left( l \right)}dV\left( m \right)}
	\\
	&=\rho \left( z \right) ^{b_1-\alpha _1}\rho \left( w \right) ^{b_2-\alpha _2}\int_{T_B}{\frac{\rho \left( l \right) ^{\beta _1+a_1+t_1}}{\rho \left( z,l \right) ^{c_1}\rho \left( l,u\mathbf{i} \right) ^{s_1}}dV\left( l \right)}\cdot \int_{T_B}{\frac{\rho \left( m \right) ^{\beta _2+a_2+t_2}}{\rho \left( w,m \right) ^{c_2}\rho \left( m,\eta \mathbf{i} \right) ^{s_2}}dV\left( m \right)}
	\\
	&=C\frac{\rho \left( z \right) ^{b_1-\alpha _1}}{\rho \left( z,u\mathbf{i} \right) ^{c_1+s_1-\left( n+1+\beta _1+a_1+t_1 \right)}}\cdot \frac{\rho \left( w \right) ^{b_2-\alpha _2}}{\rho \left( w,\eta \mathbf{i} \right) ^{c_2+s_2-\left( n+1+\beta _2+a_2+t_2 \right)}}.
	\end{aligned}
	$$
	It is easy to see that 
	$$
	\begin{aligned}
	\lVert T_{\vec{a},\ \vec{b},\ \vec{c}}^{*}\ f_{u,\eta}\left( z,w \right) \rVert _{\vec{p}^{\prime},\vec{\alpha}}&=C\left( \int_{T_B}{\left( \int_{T_B}{\left| T_{\vec{a},\ \vec{b},\ \vec{c}}^{*}\ f_{u,\eta}\left( z,w \right) \right|^{p_{1}^{\prime}}dV_{\alpha _1}\left( z \right)} \right) ^{\frac{p_{2}^{\prime}}{p_{1}^{\prime}}}dV_{\alpha _2}\left( w \right)} \right) ^{\frac{1}{p_{2}^{\prime}}}
	\\
	&=C\left( \int_{T_B}{\frac{\rho \left( z \right) ^{\left( b_1-\alpha _1 \right) p_{1}^{\prime}+\alpha _1}}{\rho \left( z,u\mathbf{i} \right) ^{\left( c_1+s_1-\left( \beta _1+t_1+a_1+n+1 \right) \right) p_{1}^{\prime}}}dV\left( z \right)} \right) ^{\frac{1}{p_{1}^{\prime}}}\\
	&\times\left( \int_{T_B}{\frac{\rho \left( w \right) ^{\left( b_2-\alpha _2 \right) p_{2}^{\prime}+\alpha _2}}{\rho \left( w,\eta \mathbf{i} \right) ^{\left( c_2+s_2-\left( \beta _2+t_2+a_2+n+1 \right) \right) p_{2}^{\prime}}}dV\left( w \right)} \right) ^{\frac{1}{p_{2}^{\prime}}}.
	\end{aligned}
	$$
	
	Since $T_{\vec{a},\ \vec{b},\ \vec{c}}^{*}\ f_{u,\eta}\left( z,w \right) \in L_{\vec{\alpha}}^{\vec{p}\prime}\left( T_B\times T_B \right)$, by Lemma \ref{lem jifendengshi} we have 
	$$
	\left( b_i-\alpha _i \right) p_{i}^{\prime}+\alpha _i>-1,
	$$
	that is 
	$$
	\alpha_i+1<p_i\left(b_i+1\right).
	$$
	
	Applying the above proof process to the operator $T_{\vec{a},\ \vec{b},\ \vec{c}}$, we can get 
	$$
	-q_ia_i<\beta _i+1
	$$ 
	for any $i\in \{1,2\}$.
	
	Next, we prove that for any $i\in \{1,2\}$, $c_i=n+b_i+1+\lambda_i$.
	
	For any $\xi,\ \eta \in T_B$, let
	$$
	f_{\xi ,\eta}\left( z,w \right) =\frac{\rho \left( \xi \right) ^{n+1+b_1-\left( n+1+\alpha _1 \right) /p_1}\rho \left( \eta \right) ^{n+1+b_2-\left( n+1+\alpha _2 \right) /p_2}}{\rho \left( z,\xi \right) ^{n+1+b_1}\rho \left( w,\eta \right) ^{n+1+b_2}}.
	$$
	Given $\alpha_1+1<p_1\left(b_1+1\right)$ and $\alpha_2+1<p_2\left(b_2+1\right)$, it follows that there exists a positive constant $C$ independent of $\xi$ and $\eta$ such that 
	$$
	\begin{aligned}
	\lVert f_{\xi ,\eta}\left( z,w \right) \rVert _{\vec{p},\vec{\alpha}}&=\left( \int_{T_B}{\left( \int_{T_B}{\left| f_{\xi ,\eta}\left( z,w \right) \right|^{p_1}dV_{\alpha _1}\left( z \right)} \right) ^{\frac{p_2}{p_1}}dV_{\alpha _2}\left( w \right)} \right) ^{\frac{1}{p_2}}
	\\
	&=\left( \int_{T_B}{\frac{\rho \left( \xi \right) ^{\left( n+1+b_1 \right) p_1-\left( n+1+\alpha _1 \right)}\rho \left( z \right) ^{\alpha _1}}{\left| \rho \left( z,\xi \right) \right|^{\left( n+1+b_1 \right) p_1}}dV\left( z \right)} \right) ^{\frac{1}{p_1}}\\
	&\times\left( \int_{T_B}{\frac{\rho \left( \eta \right) ^{\left( n+1+b_2 \right) p_2-\left( n+1+\alpha _2 \right)}\rho \left( w \right) ^{\alpha _2}}{\left| \rho \left( w,\eta \right) \right|^{\left( n+1+b_2 \right) p_2}}dV\left( w \right)} \right) ^{\frac{1}{p_2}}\\
	&\le C.
	\end{aligned}
	$$
	Notice that $c_i>0$ and $b_i> \left(\alpha_i+1\right)/p_i-1>-1$, by Lemma \ref{lem jifendengshi} we have 
	$$
	\begin{aligned}
	T_{\vec{a},\ \vec{b},\ \vec{c}}\ f_{\xi ,\eta}\left( z,w \right) &=\rho(z)^{a_1}\rho(w)^{a_2}\int_{T_B}{\int_{T_B}{\frac{\rho \left( l \right) ^{b_1}\rho \left( m \right) ^{b_2}}{\rho \left( z,l \right) ^{c_1}\rho \left( w,m \right) ^{c_2}}f_{\xi ,\eta}\left( l,m \right) dV\left( l \right)}dV\left( m \right)}
	\\
	&=\rho(z)^{a_1}\rho(w)^{a_2}\rho \left( \xi \right) ^{n+1+b_1-\left( n+1+\alpha _1 \right) /p_1}\rho \left( \eta \right) ^{n+1+b_2-\left( n+1+\alpha _2 \right) /p_2}\\
	&\times\int_{T_B}{\frac{\rho \left( l \right) ^{b_1}}{\rho \left( z,l \right) ^{c_1}\rho \left( l,\xi \right) ^{n+1+b_1}}dV\left( l \right)}\int_{T_B}{\frac{\rho \left( m \right) ^{b_2}}{\rho \left( w,m \right) ^{c_2}\rho \left( m,\eta \right) ^{n+1+b_2}}dV\left( m \right)}
	\\
	&=C\frac{\rho(z)^{a_1}\rho(w)^{a_2}\rho \left( \xi \right) ^{n+1+b_1-\left( n+1+\alpha _1 \right) /p_1}\rho \left( \eta \right) ^{n+1+b_2-\left( n+1+\alpha _2 \right) /p_2}}{\rho \left( z,\xi \right) ^{c_1}\rho \left( w,\eta \right) ^{c_2}}.
	\end{aligned}
	$$
	
	Since $T_{\vec{a},\ \vec{b},\ \vec{c}}\ f_{\xi ,\eta}\left( z,w \right)\in L_{\vec{\beta}}^{\vec{q}}\left( T_B\times T_B \right)$, there is a positive constant $C$ such that 
	$$
	\begin{aligned}
	\lVert T_{\vec{a},\ \vec{b},\ \vec{c}}\ f_{\xi ,\eta}\left( z,w \right) \rVert _{\vec{q},\vec{\beta}}&=\left( \int_{T_B}{\left( \int_{T_B}{\left| T_{\vec{a},\ \vec{b},\ \vec{c}}\ f_{\xi ,\eta}\left( z,w \right) \right|^{q_1}dV_{\beta _1}\left( z \right)} \right) ^{\frac{q_2}{q_1}}dV_{\beta _2}\left( w \right)} \right) ^{\frac{1}{q_2}}
	\\
	&=\rho \left( \xi \right) ^{n+1+b_1-\left( n+1+\alpha _1 \right) /p_1}\rho \left( \eta \right) ^{n+1+b_2-\left( n+1+\alpha _2 \right) /p_2}
	\\
	&\times \left( \int_{T_B}{\frac{\rho \left( z \right) ^{\beta _1+a_1q_1}}{\left| \rho \left( z,\xi \right) \right|^{c_1q_1}}dV\left( z \right)} \right) ^{\frac{1}{q_1}}\left( \int_{T_B}{\frac{\rho \left( w \right) ^{\beta _2+a_2q_2}}{\left| \rho \left( w,\eta \right) \right|^{c_2q_2}}dV\left( w \right)} \right) ^{\frac{1}{q_2}}\\
	&\le C.
	\end{aligned}
	$$
	By Lemma \ref{lem jifendengshi}, for any $i\in\{1,2\}$, we obtain
	$$
	n+1+b_i-\left( n+1+\alpha _i \right) /p_i=\left(c_iq_i-\left(n+1+\beta_i+a_iq_i\right)\right)/q_i
	$$
	which simplifies to
	$$
	c_i=n+b_i+1+\lambda_i.
	$$
\end{proof}

\begin{lemma}\label{lem b2}
	Let $1=p_1\le q_1<\infty$ and $1=p_2\le q_2<\infty$. If the operator $T_{\vec{a},\ \vec{b},\ \vec{c}}$ is bounded from $L_{\vec{\alpha}}^{\vec{1}}\left( T_B\times T_B \right) $ to $L_{\vec{\beta}}^{\vec{q}}\left( T_B\times T_B \right)$, 
	then for any $i\in\{1,2\}$,
	$$
	\begin{cases}
	-q_ia_i<\beta_i+1,\ \alpha _i<b_i,\\
	c_i=n+1+a_i+b_i+\lambda_i,\\
	\end{cases}
	$$
	where
	$$
	\lambda _i=\frac{n+1+\beta _i}{q_i}-\left( n+1+\alpha _i \right).
	$$
\end{lemma}

\begin{proof}	
	Similar to the proof of Lemma \ref{lem b1}. 
	
	We have
	$$
	T_{\vec{a},\ \vec{b},\ \vec{c}}^{*}\ f_{u,\eta}\left( z,w \right)
	=C\frac{\rho \left( z \right) ^{b_1-\alpha _1}}{\rho \left( z,u\mathbf{i} \right) ^{c_1+s_1-\left( n+1+\beta _1+a_1+t_1 \right)}}\cdot \frac{\rho \left( w \right) ^{b_2-\alpha _2}}{\rho \left( w,\eta \mathbf{i} \right) ^{c_2+s_2-\left( n+1+\beta _2+a_2+t_2 \right)}}.
	$$
	It is easy to see that 
	\begin{equation}\label{b 21}
	\begin{aligned}
	\lVert T_{\vec{a},\ \vec{b},\ \vec{c}}^{*}\ f_{u,\eta}\left( z,w \right) \rVert _{\infty}=&C\underset{z\in T_B}{\text{ess}\ \text{sup}}\frac{\rho \left( z \right) ^{b_1-\alpha _1}}{\rho \left( z,u\mathbf{i} \right) ^{c_1+s_1-\left( \beta _1+t_1+a_1+n+1 \right)}}
	\\
	&\cdot \underset{w\in T_B}{\text{ess}\ \text{sup}}\frac{\rho \left( w \right) ^{b_2-\alpha _2}}{\rho \left( w,\eta \mathbf{i} \right) ^{c_2+s_2-\left( \beta _2+t_2+a_2+n+1 \right)}}.
	\end{aligned}
	\end{equation}
	
	Since $T_{\vec{a},\ \vec{b},\ \vec{c}}^{*}\ f_{u,\eta}\left( z,w \right) \in L^\infty\left( T_B\times T_B \right)$, the above equation is finite. According to the conditions satisfied by $s_i$ and $t_i$, it can be obtained that the power of the denominator of the above formula (\ref{b 21}) is greater than zero, and $\rho(z)$, $\rho(w)$ can be infinite on $T_B$. Therefore, combined with the lemma \ref{lem qu bian liang}, we obtain $b_1>\alpha_1$ and $b_2>\alpha_2$.
	
	Applying the above proof process to the operator $T_{\vec{a},\ \vec{b},\ \vec{c}}$, we can get 
	$$
	-q_ia_i<\beta _i+1
	$$ 
	for any $i\in \{1,2\}$.
	
	Next, we prove that for any $i\in \{1,2\}$, $c_i=n+b_i+1+\lambda_i$.
	
	For any $\xi,\ \eta \in T_B$, let
	$$
	f_{\xi ,\eta}\left( z,w \right) =\frac{\rho \left( \xi \right) ^{b_1-\alpha _1}\rho \left( \eta \right) ^{b_2-\alpha _2 }}{\rho \left( z,\xi \right) ^{n+1+b_1}\rho \left( w,\eta \right) ^{n+1+b_2}}.
	$$
	Given $\alpha_1<b_1$ and $\alpha_2< b_2$, we can easily see that there is a positive constant $C$ independent of $\xi$ and $\eta$ such that 
	$$
	\begin{aligned}
	\lVert f_{\xi ,\eta}\left( z,w \right) \rVert _{\vec{1},\vec{\alpha}}&=\int_{T_B} \int_{T_B}{\left| f_{\xi ,\eta}\left( z,w \right) \right|dV_{\alpha _1}(z) }dV_{\alpha _2}(w)	\\
	&=\int_{T_B}{\frac{\rho \left( \xi \right) ^{b_1-\alpha _1}\rho \left( z \right) ^{\alpha _1}}{\left| \rho \left( z,\xi \right) \right|^{ n+1+b_1}}dV\left( z \right)}
	\int_{T_B}{\frac{\rho \left( \eta \right) ^{b_2-\alpha _2}\rho \left( w \right) ^{\alpha _2}}{\left| \rho \left( w,\eta \right) \right|^{ n+1+b_2 }}dV\left( w \right)}\\
	&\le C
	\end{aligned}
	$$
	Notice that $c_i>0$ and $b_i>\alpha_i>-1$, by Lemma \ref{lem jifendengshi} we have 
	$$
	\begin{aligned}
	T_{\vec{a},\ \vec{b},\ \vec{c}}\ f_{\xi ,\eta}\left( z,w \right) &=\rho(z)^{a_1}\rho(w)^{a_2}\int_{T_B}{\int_{T_B}{\frac{\rho \left( l \right) ^{b_1}\rho \left( m \right) ^{b_2}}{\rho \left( z,l \right) ^{c_1}\rho \left( w,m \right) ^{c_2}}f_{\xi ,\eta}\left( l,m \right) dV\left( l \right)}dV\left( m \right)}
	\\
	&=\rho(z)^{a_1}\rho \left( \xi \right) ^{b_1-\alpha _1}\int_{T_B}{\frac{\rho \left( l \right) ^{b_1}}{\rho \left( z,l \right) ^{c_1}\rho \left( l,\xi \right) ^{n+1+b_1}}dV\left( l \right)}\\
	&\times\rho(w)^{a_2}\rho \left( \eta \right) ^{b_2-\alpha _2 }\int_{T_B}{\frac{\rho \left( m \right) ^{b_2}}{\rho \left( w,m \right) ^{c_2}\rho \left( m,\eta \right) ^{n+1+b_2}}dV\left( m \right)}
	\\
	&=C\frac{\rho(z)^{a_1}\rho(w)^{a_2}\rho \left( \xi \right) ^{b_1-\alpha _1}\rho \left( \eta \right) ^{b_2-\alpha _2}}{\rho \left( z,\xi \right) ^{c_1}\rho \left( w,\eta \right) ^{c_2}}.
	\end{aligned}
	$$

	Since $T_{\vec{a},\ \vec{b},\ \vec{c}}\ f_{\xi ,\eta}\left( z,w \right)\in L_{\vec{\beta}}^{\vec{q}}\left( T_B\times T_B \right)$, we know that there is a positive constant $C$ such that 
	$$
	\begin{aligned}
	\lVert T_{\vec{a},\ \vec{b},\ \vec{c}}\ f_{\xi ,\eta}\left( z,w \right) \rVert _{\vec{q},\vec{\beta}}&=\left( \int_{T_B}{\left( \int_{T_B}{\left| T_{\vec{a},\ \vec{b},\ \vec{c}}\ f_{\xi ,\eta}\left( z,w \right) \right|^{q_1}dV_{\beta _1}\left( z \right)} \right) ^{\frac{q_2}{q_1}}dV_{\beta _2}\left( w \right)} \right) ^{\frac{1}{q_2}}
	\\
	&=\rho \left( \xi \right) ^{b_1-\alpha _1}
	\left( \int_{T_B}{\frac{\rho \left( z \right) ^{a_1q_1+\beta _1}}{\left| \rho \left( z,\xi \right) \right|^{c_1q_1}}dV\left( z \right)} \right) ^{\frac{1}{q_1}}\\
	&\times\rho \left( \eta \right) ^{b_2-\alpha _2}\left( \int_{T_B}{\frac{\rho \left( w \right) ^{a_2q_2+\beta _2}}{\left| \rho \left( w,\eta \right) \right|^{c_2q_2}}dV\left( w \right)} \right) ^{\frac{1}{q_2}}\\
	&\le C.
	\end{aligned}
	$$
	By Lemma \ref{lem jifendengshi}, for any $i\in\{1,2\}$, we have
	$$
	b_i-\alpha _i=\left(c_iq_i-\left(n+1+a_iq_i+\beta_i\right)\right)/q_i,
	$$
	which simplifies to
	$$
	c_i=n+b_i+1+\lambda_i,
	$$
	where 
	$$
	\lambda _i=\frac{n+1+\beta _i}{q_i}-\left( n+1+\alpha _i \right).
	$$
\end{proof}

\begin{lemma}\label{lem b3}
	Let $1<p_1\le q_1<\infty$ and $1=p_2\le q_2<\infty$. If the operator $T_{\vec{a},\ \vec{b},\ \vec{c}}$ is bounded from $L_{\vec{\alpha}}^{\vec{p}}\left( T_B\times T_B \right) $ to $L_{\vec{\beta}}^{\vec{q}}\left( T_B\times T_B \right)$, 
	then for any $i\in\{1,2\}$,
	$$
	\begin{cases}
	\alpha _1+1<p_1\left( b_1+1 \right) ,\ \alpha _2<b_2,\\
	-q_ia_i<\beta_i+1,\ c_i=n+1+a_i+b_i+\lambda_i,\\
	\end{cases}
	$$
	where
	$$
	\begin{cases}
	\lambda _1=\frac{n+1+\beta _1}{q_1}-\frac{n+1+\alpha _1}{p_1},\\
	\lambda _2=\frac{n+1+\beta _2}{q_2}-\left( n+1+\alpha _2 \right).\\
	\end{cases}
	$$
\end{lemma}
\begin{proof}
	Similar to the proof of Lemma \ref{lem b1}. 
	We get 
	$$
	T_{\vec{a},\ \vec{b},\ \vec{c}}^{*}\ f_{u,\eta}\left( z,w \right)
	=C\frac{\rho \left( z \right) ^{b_1-\alpha _1}}{\rho \left( z,u\mathbf{i} \right) ^{c_1+s_1-\left( n+1+\beta _1+a_1+t_1 \right)}}\cdot \frac{\rho \left( w \right) ^{b_2-\alpha _2}}{\rho \left( w,\eta \mathbf{i} \right) ^{c_2+s_2-\left( n+1+\beta _2+a_2+t_2 \right)}}.
	$$
	It is easy to see that 
	$$
	\begin{aligned}
	\lVert T_{\vec{a},\ \vec{b},\ \vec{c}}^{*}\ f_{u,\eta}\left( z,w \right) \rVert _{\vec{p}^{\prime},\vec{\alpha}}&=C\underset{w\in T_B}{\text{ess}\ \text{sup}}\left( \int_{T_B}{\left| T_{\vec{a},\ \vec{b},\ \vec{c}}^{*}\ f_{u,\eta}\left( z,w \right) \right|^{p_{1}^{\prime}}dV_{\alpha _1}\left( z \right)} \right) ^{\frac{1}{p_{1}^{\prime}}}
	\\
	&\times\left( \int_{T_B}{\frac{\rho \left( z \right) ^{\left( b_1-\alpha _1 \right) p_{1}^{\prime}+\alpha _1}}{\rho \left( z,u\mathbf{i} \right) ^{\left( c_1+s_1-\left( \beta _1+t_1+a_1+n+1 \right) \right) p_{1}^{\prime}}}dV\left( z \right)} \right) ^{\frac{1}{p_{1}^{\prime}}}\\
	&=C\underset{w\in T_B}{\text{ess}\ \text{sup}}\frac{\rho \left( w \right) ^{b_2-\alpha _2}}{\rho \left( w,\eta \mathbf{i} \right) ^{c_2+s_2-\left( \beta _2+t_2+a_2+n+1 \right)}}.
	\end{aligned}
	$$
	
	Since $T_{\vec{a},\ \vec{b},\ \vec{c}}^{*}\ f_{u,\eta}\left( z,w \right) \in L_{\vec{\alpha}}^{\vec{p}^{\prime}}\left( T_B\times T_B \right)$, by Lemma \ref{lem jifendengshi} we have 
	$$
	\left( b_1-\alpha _1 \right) p_{1}^{\prime}+\alpha _1>-1,
	$$
	that is 
	$$
	\alpha_1+1<p_1\left(b_1+1\right).
	$$
	According to $s_2-t_2>\beta _2+a_2-c_2+n+1$, the power of the denominator of $\rho \left( w \right) ^{b_2-\alpha _2}/\rho \left( w,\eta \mathbf{i} \right) ^{c_2+s_2-\left( \beta _2+t_2+n+1 \right)}$ is greater than 0, and $\rho(w)$ can be infinite on $T_B$. Combined with Lemma \ref{lem qu bian liang}, the above finite means $b_2>\alpha_2$.	
	
	Applying the above proof process to the operator $T_{\vec{a},\ \vec{b},\ \vec{c}}$, we can get 
	$$
	-q_ia_i<\beta _i+1
	$$ 
	for any $i\in \{1,2\}$.
	
	Next, we prove that for any $i\in \{1,2\}$, $c_i=n+b_i+1+\lambda_i$.
	
	For any $\xi,\ \eta \in T_B$, let
	$$
	f_{\xi ,\eta}\left( z,w \right) =\frac{\rho \left( \xi \right) ^{n+1+b_1-\left( n+1+\alpha _1 \right) /p_1}\rho \left( \eta \right) ^{b_2-\alpha _2 }}{\rho \left( z,\xi \right) ^{n+1+b_1}\rho \left( w,\eta \right) ^{n+1+b_2}}.
	$$
	Given $\alpha_1+1<p_1\left(b_1+1\right)$ and $\alpha_2< b_2$, we can easily see that there is a positive constant $C$ independent of $\xi$ and $\eta$ such that 
	$$
	\begin{aligned}
	\lVert f_{\xi ,\eta}\left( z,w \right) \rVert _{\vec{p},\vec{\alpha}}&=\int_{T_B}{\left( \int_{T_B}{\left| f_{\xi ,\eta}\left( z,w \right) \right|^{p_1}dV_{\alpha _1}\left( z \right)} \right) ^{\frac{1}{p_1}}dV_{\alpha _2}\left( w \right)}
	\\
	&=\int_{T_B}{\frac{\rho \left( \xi \right) ^{\left( n+1+b_1 \right) p_1-\left( n+1+\alpha _1 \right)}\rho \left( z \right) ^{\alpha _1}}{\left| \rho \left( z,\xi \right) \right|^{\left( n+1+b_1 \right) p_1}}dV\left( z \right)}
	\int_{T_B}{\frac{\rho \left( \eta \right) ^{b_2-\alpha _2}\rho \left( w \right) ^{\alpha _2}}{\left| \rho \left( w,\eta \right) \right|^{ n+1+b_2 }}dV\left( w \right)}\\
	&\le C.
	\end{aligned}
	$$
	Notice that $c_i>0$, $b_1> \left(\alpha+1\right)/p-1>-1$ and $b_2>\alpha_2>-1$, by Lemma \ref{lem jifendengshi} we have 
	$$
	\begin{aligned}
	T_{\vec{a},\ \vec{b},\ \vec{c}}\ f_{\xi ,\eta}\left( z,w \right) &=\rho(z)^{a_1}\rho(w)^{a_2}\int_{T_B}{\int_{T_B}{\frac{\rho \left( l \right) ^{b_1}\rho \left( m \right) ^{b_2}}{\rho \left( z,l \right) ^{c_1}\rho \left( w,m \right) ^{c_2}}f_{\xi ,\eta}\left( l,m \right) dV\left( l \right)}dV\left( m \right)}
	\\
	&=C\frac{\rho(z)^{a_1}\rho(w)^{a_2}\rho \left( \xi \right) ^{n+1+b_1-\left( n+1+\alpha _1 \right) /p_1}\rho \left( \eta \right) ^{b_2-\alpha _2}}{\rho \left( z,\xi \right) ^{c_1}\rho \left( w,\eta \right) ^{c_2}}.
	\end{aligned}
	$$
	Since $T_{\vec{a},\ \vec{b},\ \vec{c}}\ f_{\xi ,\eta}\left( z,w \right)\in L_{\vec{\beta}}^{\vec{q}}\left( T_B\times T_B \right)$, we know that there is a positive constant $C$ such that 
	$$
	\begin{aligned}
	\lVert T_{\vec{a},\ \vec{b},\ \vec{c}}\ f_{\xi ,\eta}\left( z,w \right) \rVert _{\vec{q},\vec{\beta}}&=\left( \int_{T_B}{\left( \int_{T_B}{\left| T_{\vec{a},\ \vec{b},\ \vec{c}}\ f_{\xi ,\eta}\left( z,w \right) \right|^{q_1}dV_{\beta _1}\left( z \right)} \right) ^{\frac{q_2}{q_1}}dV_{\beta _2}\left( w \right)} \right) ^{\frac{1}{q_2}}
	\\
	&=\rho \left( \xi \right) ^{n+1+b_1-\left( n+1+\alpha _1 \right) /p_1}\rho \left( \eta \right) ^{b_2-\alpha _2}
	\\
	&\times \left( \int_{T_B}{\frac{\rho \left( z \right) ^{a_1q_1+\beta _1}}{\left| \rho \left( z,\xi \right) \right|^{c_1q_1}}dV\left( z \right)} \right) ^{\frac{1}{q_1}}\left( \int_{T_B}{\frac{\rho \left( w \right) ^{a_2q_2+\beta _2}}{\left| \rho \left( w,\eta \right) \right|^{c_2q_2}}dV\left( w \right)} \right) ^{\frac{1}{q_2}}\\
	&\le C.
	\end{aligned}
	$$
	Hence, by Lemma \ref{lem jifendengshi} we have
	$$
	n+1+b_1-\left( n+1+\alpha _1 \right) /p_1=\left(c_1q_1-\left(n+1+a_1q_1+\beta_1\right)\right)/q_1
	$$
	and
	$$
	b_2-\alpha _2=\left(c_2q_2-\left(n+1+a_2q_2+\beta_2\right)\right)/q_2,
	$$
	that is,
	$$
	c_1=n+1+a_1+b_1+\frac{n+1+\beta _1}{q_1}-\frac{n+1+\alpha _1}{p_1}
	$$
	and
	$$
	c_2=a_2+b_2-\alpha_2+\frac{n+1+\beta _2}{q_2}.
	$$
	
\end{proof}

\begin{lemma}\label{lem b4}
	Let $1=p_1\le q_1<\infty$ and $1<p_2\le q_2<\infty$. If the operator $T_{\vec{a},\ \vec{b},\ \vec{c}}$ is bounded from $L_{\vec{\alpha}}^{\vec{p}}\left( T_B\times T_B \right) $ to $L_{\vec{\beta}}^{\vec{q}}\left( T_B\times T_B \right)$, 
	then for any $i\in\{1,2\}$,
	$$
	\begin{cases}
	 \alpha _1<b_1,\ \alpha _2+1<p_2\left( b_2+1 \right) ,\\
	-q_ia_i<\beta_i+1,\ c_i=n+1+a_i+b_i+\lambda_i,\\
	\end{cases}
	$$
	where
	$$
	\begin{cases}
	\lambda _1=\frac{n+1+\beta _1}{q_1}-\left( n+1+\alpha _1 \right),\\
	\lambda _2=\frac{n+1+\beta _2}{q_2}-\frac{n+1+\alpha _2}{p_2}.\\
	\end{cases}
	$$
\end{lemma}

\begin{proof}
	The proof of  Lemma \ref{lem b4} is similar to the proof of Lemma \ref{lem b3}, which only needs to modify the corresponding norm appropriately, thus we omit its proof.
\end{proof}

\begin{lemma}\label{lem b5}
	Let $1<p_1<q_1=\infty$ and $1<p_2<q_2=\infty$. Suppose $T_{\vec{a},\ \vec{b},\ \vec{c}}$ is bounded from $L_{\vec{\alpha}}^{\vec{p}}\left( T_B\times T_B \right) $ to $L^{\infty}\left( T_B\times T_B \right)$, then for any $i\in\{1,2\}$,
	$$
	\begin{cases}
	a_i>0,\ \alpha _i+1<p_i(b_i+1),\\
	c_i=n+1+a_i+b_i+\lambda_i,\\
	\end{cases}
	$$
	where
	$$
	\lambda _i=-\frac{n+1+\alpha _i}{p_i}.
	$$
\end{lemma}

\begin{proof}
	
	Since $T_{\vec{a},\ \vec{b},\ \vec{c}}$ is bounded, it follows that for any fixed $z,w\in T_B$, the integral  
	$$
	\begin{aligned}
	T_{\vec{a},\ \vec{b},\ \vec{c}}\,\,f\left( z,w \right) &=\rho(z)^{a_1}\rho(w)^{a_2}\int_{T_B}{\int_{T_B}{\frac{\rho \left( u \right) ^{b_1}\rho \left( \eta \right) ^{b_2}}{\rho \left( z,u \right) ^{c_1}\rho \left( w,\eta \right) ^{c_2}}f\left( u,\eta \right) dV\left( u \right)}dV\left( \eta \right)}
	\\
	&=\rho(z)^{a_1}\rho(w)^{a_2}\int_{T_B}{\int_{T_B}{\frac{\rho \left( u \right) ^{b_1-\alpha _1}\rho \left( \eta \right) ^{b_2-\alpha _2}}{\rho \left( z,u \right) ^{c_1}\rho \left( w,\eta \right) ^{c_2}}f\left( u,\eta \right) dV_{\alpha _1}\left( u \right)}dV_{\alpha _2}\left( \eta \right)}
	\end{aligned}
	$$
	is finte for each $f\in L_{\vec{\alpha}}^{\vec{p}}\left( T_B\times T_B \right).$
	
	By duality, we get that for any fixed $z, w\in T_B$,
	$$
	\frac{\rho(z)^{a_1}\rho \left( u \right) ^{b_1-\alpha _1}\rho(w)^{a_2}\rho \left( \eta \right) ^{b_2-\alpha _2}}{\rho \left( z,u \right) ^{c_1}\rho \left( w,\eta \right) ^{c_2}}\in L_{\vec{\alpha}}^{\vec{p}\prime}\left( T_B\times T_B \right).
	$$
	Therefore, by calculation, we have 
	$$
	\begin{aligned}
	&\rho(z)^{a_1}\rho(w)^{a_2}\left( \int_{T_B}{\frac{\rho \left( u \right) ^{p_{1}^{\prime}\left( b_1-\alpha _1 \right) +\alpha _1}}{\left| \rho \left( z,u \right) \right|^{c_1p_{1}^{\prime}}}dV\left( u \right)} \right) ^{\frac{1}{p_{1}^{\prime}}}\left( \int_{T_B}{\frac{\rho \left( \eta \right) ^{p_{2}^{\prime}\left( b_2-\alpha _2 \right) +\alpha _2}}{\left| \rho \left( w,\eta \right) \right|^{c_2p_{2}^{\prime}}}dV\left( \eta \right)} \right) ^{\frac{1}{p_{2}^{\prime}}}\\
	&\lesssim\frac{\rho \left( z \right) ^{a_1}\rho \left( w \right) ^{a2}}{\rho \left( z \right) ^{c_1+\alpha _1-b_1-\left( n+1+\alpha _1 \right) /p_{1}^{\prime}}\rho \left( w \right) ^{c_2+\alpha _2-b_2-\left( n+1+\alpha _2 \right) /p_{2}^{\prime}}}\\
	&<\infty.
	\end{aligned}
	$$
	
	By Lemma \ref{lem jifendengshi}, we know that for any $i\in \{1,2\}$,
	$$
	\begin{cases}
	p_{i}^{\prime}\left( b_i-\alpha _i \right) +\alpha _i>-1, \\
	c_ip_{i}^{\prime}>n+1+p_{i}^{\prime}\left( b_i-\alpha _i \right) +\alpha _i,\\
	\end{cases}
	$$
	that is 
	$$
	\begin{cases}
	\alpha _i+1<p_i(b_i+1),\\
	c_i>n+1+b_i-\frac{n+1+\alpha_i}{p_i}.\\
	\end{cases}
	$$
	
	By the arbitrariness of $z$ and $w$, we have
	$$
	\begin{cases}
	a_i>0,\\
	c_i=n+1+a_i+b_i-\frac{n+1+\alpha _i}{p_i}.\\
	\end{cases}
	$$
	
	This completes the proof of the lemma.
	
\end{proof}

\begin{lemma}\label{lem b6}
	Let $1=p_1<q_1=\infty$ and $1=p_2<q_2=\infty$. Suppose $T_{\vec{a},\ \vec{b},\ \vec{c}}$ is bounded from $L_{\vec{\alpha}}^{\vec{1}}\left( T_B\times T_B \right) $ to $L^{\infty}\left( T_B\times T_B \right)$, then for any $i\in\{1,2\}$,
	$$
	\begin{cases}
	a_i\ge0,\ b_i\ge\alpha _i,\\
	c_i=n+1+a_i+b_i+\lambda_i,\\
	\end{cases}
	$$
	where
	$$
	\lambda _i=-n-1-\alpha _i.
	$$
\end{lemma}

\begin{proof}
	Similar to the proof of Lemma \ref{lem b5}, by duality, we obtain that for any fixed $z, w\in T_B$,
	\begin{equation}\label{lem b6 t1}
	\mathop {\mathrm{ess}\ \mathrm{sup}}_{\left( u,\eta \right) \in T_B\times T_B}\frac{\rho \left( z \right) ^{a_1}\rho \left( u \right) ^{b_1-\alpha _1}\rho \left( w \right) ^{a_2}\rho \left( \eta \right) ^{b_2-\alpha _2}}{\left| \rho \left( z,u \right) \right|^{c_1}\left| \rho \left( w,\eta \right) \right|^{c_2}}
	<\infty.
	\end{equation}
	
	For any $i\in \{1,2\}$, we first prove that $ c_i \ge 0 $. Suppose $ c_i < 0 $. Then we take $z = w = (0', i) $, $ u = (0', x_n + i)$ and $ \eta = (0', y_n + i) $ such that $\left| x_n \right|\rightarrow \infty$ and $\left| y_n \right|\rightarrow \infty$ as $n\rightarrow \infty$. 
	Thus, for any fixed $z, w\in T_B$, we have  
	$$
	\mathop {\mathrm{ess}\ \mathrm{sup}}_{\left( u,\eta \right) \in T_B\times T_B}\frac{\rho \left( z \right) ^{a_1}\rho \left( u \right) ^{b_1-\alpha _1}\rho \left( w \right) ^{a_2}\rho \left( \eta \right) ^{b_2-a_2}}{\left| \rho \left( z,u \right) \right|^{c_1}\left| \rho \left( w,\eta \right) \right|^{c_2}}\gtrsim \left| x_n \right|^{-c_1}\left| y_n \right|^{-c_2}\rightarrow \infty ,
	$$
	which contradicts to (\ref{lem b6 t1}).
	
	Next, we prove that $b_i\ge \alpha_i$ and $c_i\ge b_i-\alpha_i+a_i$. Similar to the proof of the previous Lemma \ref{lem b2}, it is easy to get $b_i\ge \alpha_i$ and $c_i\ge b_i-\alpha_i$.
	From the arbitrariness of $z$ and $w$ in $T_B$, we get $c_i= b_i-\alpha_i+a_i$ for any $i\in\{1,2\}$.
\end{proof}

\begin{lemma}\label{lem b7}
	Let $1=p_1<q_1=\infty$ and $1<p_2<q_2=\infty$. Suppose $T_{\vec{a},\ \vec{b},\ \vec{c}}$ is bounded from $L_{\vec{\alpha}}^{\vec{p}}\left( T_B\times T_B \right) $ to $L^{\infty}\left( T_B\times T_B \right)$, then for any $i\in\{1,2\}$,
	$$
	\begin{cases}
	a_1\ge0,\ a_2>0,\\
	b_1\ge \alpha_1,\ p_2\left( b_2+1 \right) >\alpha _2+1,\\
	c_i=n+1+a_i+b_i+\lambda_i,\\
	\end{cases}
	$$
	where
	$$
	\begin{cases}
	\lambda _1=-n-1-\alpha _1,\\
	\lambda_2=-\frac{n+1+\alpha _2}{p_2}.\\
	\end{cases}
	$$		
\end{lemma}

\begin{proof}
	Similar to the method of lemma \ref{lem b5}, by duality, we have for any fixed $z, w\in T_B$,
	$$
	\frac{\rho(z)^{a_1}\rho \left( u \right) ^{b_1-\alpha _1}\rho(w)^{a_2}\rho \left( \eta \right) ^{b_2-\alpha _2}}{\rho \left( z,u \right) ^{c_1}\rho \left( w,\eta \right) ^{c_2}}\in L_{\vec{\alpha}}^{\vec{p}\prime}\left( T_B\times T_B \right),
	$$
	that is 
	$$
	\rho \left( z \right) ^{a_1}\rho \left( w \right) ^{a_2}\underset{u\in T_B}{\text{ess}\ \text{sup}}\frac{\rho \left( u \right) ^{b_1-\alpha _1}}{\left| \rho \left( z,u \right) \right|^{c_1}}\left( \int_{T_B}{\frac{\rho \left( \eta \right) ^{p_{2}^{\prime}\left( b_2-\alpha _2 \right) +\alpha _2}}{\left| \rho \left( w,\eta \right) \right|^{c_2p_{2}^\prime}}dV\left( \eta \right)} \right) ^{\frac{1}{p_{2}^\prime}}<\infty .
	$$
	
	Similar to the method of Lemma \ref{lem b5} and Lemma \ref{lem b6}, we can complete the proof of the lemma.
\end{proof}

\begin{lemma}\label{lem b8}
	Let $1<p_1<q_1=\infty$ and $1=p_2<q_2=\infty$. Suppose $T_{\vec{a},\ \vec{b},\ \vec{c}}$ is bounded from $L_{\vec{\alpha}}^{\vec{p}}\left( T_B\times T_B \right) $ to $L^{\infty}\left( T_B\times T_B \right)$, then for any $i\in\{1,2\}$,
	$$
	\begin{cases}
	a_1>0,\ a_2\ge0, \\
	p_1\left( b_1+1 \right) >\alpha _1+1,\ b_2\ge \alpha_2,\\
	c_i=n+1+a_i+b_i+\lambda_i,\\
	\end{cases}
	$$
	where
	$$
	\begin{cases}
	\lambda_1=-\frac{n+1+\alpha _1}{p_1},\\
	\lambda _2=-n-1-\alpha _2.\\
	\end{cases}
	$$
\end{lemma}

\begin{proof}
	Lemma \ref{lem b8} is the symmetric case of Lemma \ref{lem b7}. Thus,  we omit the proof here.	
\end{proof}

\begin{lemma}\label{lem b9}
	Let $p_1=q_1=\infty$ and $p_2=q_2=\infty$. Suppose $T_{\vec{a},\ \vec{b},\ \vec{c}}$ is bounded from $L^{\infty}\left( T_B\times T_B \right) $ to $L^{\infty}\left( T_B\times T_B \right)$, then for any $i\in \{1,2\}$,
	$$
	\begin{cases}
	a_i>0,\ b_i>-1,\\
	c_i=n+1+a_i+b_i.\\
	\end{cases}
	$$
\end{lemma}

\begin{proof}
	For any $f\in L^\infty(T_B\times T_B)$, the boundedness of $T_{\vec{a},\ \vec{b},\ \vec{c}}:L^{\infty}\left( T_B\times T_B \right) \rightarrow L^{\infty}\left( T_B\times T_B \right)$ implies that
	\begin{equation}\label{lem b9 t1}
	\rho \left( z \right) ^{a_1}\rho \left( w \right) ^{a_2}\int_{T_B}{\int_{T_B}{\frac{\rho \left( u \right) ^{b_1}\rho \left( \eta \right) ^{b_2}}{\left| \rho \left( z,u \right) \right|^{c_1}\left| \rho \left( w,\eta \right) \right|^{c_2}}dV\left( u \right)}dV\left( \eta \right)}<\infty ,
	\end{equation}
	for any fixed $z,w\in T_B$.
	
	It follows from Lemma \ref{lem jifendengshi} that $b_i>-1$ and $c_i>n+1+b_i$ for any $i\in \{1,2\}$. 	
	(\ref{lem b9 t1}) becomes as follows :
	$$
	\frac{\rho \left( z \right) ^{a_1}\rho \left( w \right) ^{a2}}{\rho \left( z \right) ^{c_1-\left(b_1+ n+1\right)}\rho \left( w \right) ^{c_2-\left(b_2+ n+1\right)}}<\infty .
	$$
	
	By the arbitrariness of $z$ and $w$, this means $a_i>0$ and $c_i=n+1+a_i+b_i$ for any $i\in \{1,2\}$.
\end{proof}

\begin{lemma}\label{lem b10}
	Let $1=p_1<q_1=\infty$ and $p_2=q_2=\infty$. Suppose $T_{\vec{a},\ \vec{b},\ \vec{c}}$ is bounded from $L_{\vec{\alpha}}^{\vec{p}}\left( T_B\times T_B \right) $ to $L^{\infty}\left( T_B\times T_B \right)$, then for any $i\in\{1,2\}$,
	$$
	\begin{cases}
	a_1\ge0,\ a_2>0, \\
	b_1 \ge \alpha _1,\ b_2>-1,\\
	c_i=n+1+a_i+b_i+\lambda_i,\\
	\end{cases}
	$$
	where
	$$
	\begin{cases}
	\lambda _1=-n-1-\alpha _1,\\
	\lambda_2=0.\\
	\end{cases}
	$$
	
\end{lemma}

\begin{proof}
	Since $p_1=1$, $p_2=\infty$, we can conclude that for any $f\in L_{\vec{\alpha}}^{\vec{p}}\left( T_B\times T_B \right)$, there exists a positive constant C and a function $g$ such that $ \left| f\left( z,w \right) \right|\le C\left| g\left( z \right) \right|$, where  $g$ satisfies 
	$$
	\int_{T_B}{\left| g\left( z \right) \right|dV_{\alpha _1}\left( z \right)}<\infty .
	$$
	Since $T_{\vec{a},\ \vec{b},\ \vec{c}}$ is bounded, it follows that for any fixed $z,w\in T_B$, the integral  

	\begin{equation}
	\begin{aligned}
	\left| T_{\vec{a},\  \vec{b},\  \vec{c}}\,\,f\left( z,w \right) \right|&\le \rho \left( z \right) ^{a_1}\rho \left( w \right) ^{a_2}\int_{T_B}{\int_{T_B}{\frac{\rho \left( u \right) ^{b_1-\alpha _1}\rho \left( \eta \right) ^{b_2-\alpha _2}}{\left| \rho \left( z,u \right) \right|^{c_1}\left| \rho \left( w,\eta \right) \right|^{c_2}}\left| g\left( u \right) \right|dV_{\alpha _1}\left( u \right)}dV_{\alpha _2}\left( \eta \right)}
	\\
	&=\rho \left( z \right) ^{a_1}\int_{T_B}{\frac{\rho \left( u \right) ^{b_1-\alpha _1}}{\left| \rho \left( z,u \right) \right|^{c_1}}\left| g\left( u \right) \right|dV_{\alpha _1}\left( u \right)} \\
	&\times \rho \left( w \right) ^{a_2}\int_{T_B}{\frac{\rho \left( \eta \right) ^{b_2-\alpha _2}}{\left| \rho \left( w,\eta \right) \right|^{c_2}}dV_{\alpha _2}\left( \eta \right)}
	\end{aligned}
	\end{equation}
	is finte for each $f\in L_{\vec{\alpha}}^{\vec{p}}\left( T_B\times T_B \right).$
	
	The proof of the remaining part is similar to the proof of Lemma \ref{lem b6} and Lemma \ref{lem b9}, and we will not repeat it.
\end{proof}

\begin{lemma}\label{lem b11}
	Let $p_1=q_1=\infty$ and $1=p_2<q_2=\infty$. Suppose $T_{\vec{a},\ \vec{b},\ \vec{c}}$ is bounded from $L_{\vec{\alpha}}^{\vec{p}}\left( T_B\times T_B \right) $ to $L^{\infty}\left( T_B\times T_B \right)$, then for any $i\in \{1,2\}$,
	$$
	\begin{cases}
	a_1>0,\ a_2\ge 0,\\
	b_1>-1,\  b_2\ge \alpha_2,\\
	c_i=n+1+a_i+b_i+\lambda_i.\\
	\end{cases}
	$$
	where
	$$
	\begin{cases}
	\lambda_1=0,\\
	\lambda _2=-n-1-\alpha _2.\\
	\end{cases}
	$$
\end{lemma}

\begin{proof}
	Lemma \ref{lem b10} is the symmetric case of Lemma \ref{lem b9}. Thus,  we omit the proof here.
\end{proof}

\begin{lemma}\label{lem b12}
	Let $p_1=q_1=\infty$ and $1<p_2<q_2=\infty$. Suppose $T_{\vec{a},\ \vec{b},\ \vec{c}}$ is bounded from $L_{\vec{\alpha}}^{\vec{p}}\left( T_B\times T_B \right) $ to $L^{\infty}\left( T_B\times T_B \right)$, then for any $i\in \{1,2\}$,
	$$
	\begin{cases}
	a_1>0,\ a_2>0,\\
	b_1>-1,\ p_2\left( b_2+1 \right) >\alpha _2+1,\\
	c_i=n+1+a_i+b_i+\lambda _i,\\
	\end{cases}	
	$$
	where
	$$
	\begin{cases}
	\lambda _1=0,\\
	\lambda_2=-\frac{n+1+\alpha _2}{p_2}.\\
	\end{cases}
	$$
\end{lemma}

\begin{proof}
	The proof of this lemma is similar to Lemma \ref{lem b10}. From the known exponential condition, we know that there is a positive constant $C$ and function $g$ such that $\left| f\left( z,w \right) \right|\le Cg\left( w \right) $, where $g$ satisfies the following equation:
	$$
	\left( \int_{T_B}{\left| g\left( w \right) \right|^{p_2}dV_{\alpha _2}\left( z \right)} \right) ^{\frac{1}{p_2}}<\infty .
	$$	
	The boundedness of $T_{\vec{a},\ \vec{b},\ \vec{c}}:L_{\vec{\alpha}}^{\vec{p}}\left( T_B\times T_B \right) \rightarrow L^{\infty}\left( T_B\times T_B \right)$ implies that
	$$
	\begin{aligned}
	\left| T_{\vec{a},\ \vec{b},\ \vec{c}}\,\,f\left( z,w \right) \right|&\le \rho \left( z \right) ^{a_1}\rho \left( w \right) ^{a_2}\int_{T_B}{\int_{T_B}{\frac{\rho \left( u \right) ^{b_1}\rho \left( \eta \right) ^{b_2}}{\left| \rho \left( z,u \right) \right|^{c_1}\left| \rho \left( w,\eta \right) \right|^{c_2}}\left| g\left( \eta \right) \right|dV\left( u \right)}dV\left( \eta \right)}
	\\
	&=\rho \left( z \right) ^{a_1}\int_{T_B}{\frac{\rho \left( u \right) ^{b_1}}{\left| \rho \left( z,u \right) \right|^{c_1}}dV\left( u \right)}\times \rho \left( w \right) ^{a_2}\int_{T_B}{\frac{\rho \left( \eta \right) ^{b_2-\alpha _2}}{\left| \rho \left( w,\eta \right) \right|^{c_2}}\left| g\left( \eta \right) \right|dV_{\alpha _2}\left( \eta \right)}
	\end{aligned}
	$$
	
	Similar to the method of Lemma \ref{lem b9} and Lemma \ref{lem b5}, and we will not repeat it.
\end{proof}

\begin{lemma}\label{lem b13}
	Let $1<p_1<q_1=\infty$ and $p_2=q_2=\infty$. Suppose $T_{\vec{a},\ \vec{b},\ \vec{c}}$ is bounded from $L_{\vec{\alpha}}^{\vec{p}}\left( T_B\times T_B \right) $ to $L^{\infty}\left( T_B\times T_B \right)$, then for any $i\in \{1,2\}$,
	$$
	\begin{cases}
	a_1>0,\ a_2>0,\\
	p_1\left( b_1+1 \right) >\alpha _1+1,\ b_2>-1,\\
	c_i=n+1+a_i+b_i+\lambda_i,\\
	\end{cases}	
	$$
	where
	$$
	\begin{cases}
	\lambda_1=-\frac{n+1+\alpha _2}{p_2},\\
	\lambda _2=0.\\
	\end{cases}
	$$
\end{lemma}

\begin{proof}
	Lemma \ref{lem b13} is the symmetric case of Lemma \ref{lem b12}. Thus,  we omit the proof here.
\end{proof}

\section{The Sufficiency for the Boundedness of $S_{\vec{a},\ \vec{b},\ \vec{c}}$}
\ \ \ \
In this section, we give the necessary conditions for the boundedness of the operator $S_{\vec{a},\ \vec{b},\ \vec{c}}$, and the conclusions are the following lemmas \ref{lem 1 1} to lemma \ref{lem 1 13}.
\begin{lemma}\label{lem 1 1}
	Let $1<p_-\le p_+\le q_-<\infty$. If the parameters satisfy for any $i \in \{1,2\}$, 
	$$
	\begin{cases}
	-q_ia_i<\beta _i+1,\  \alpha _i+1<p_i\left( b_i+1 \right) ,\\
	c_i=n+1+a_i+b_i+\lambda_i,\\
	\end{cases}
	$$
	\\
	where 
	$$
	\lambda_i =\frac{n+1+\beta _i}{q_i}-\frac{n+1+\alpha _i}{p_i}.
	$$	
	Then the operator $S_{\vec{a},\ \vec{b},\ \vec{c}}$ is bounded from $L_{\vec{\alpha}}^{\vec{p}}\left( T_B\times T_B \right)$ to $L_{\vec{\beta}}^{\vec{q}}\left( T_B\times T_B \right)$.
\end{lemma}

\begin{proof}
	For any $i\in \{1,2\}$, let $\tau _i=c_i-a_i-b_i+\alpha _i=\frac{n+1+\alpha _i}{p_{i}^{\prime}}+\frac{n+1+\beta _i}{q_i}>0$. 
	
	Since $-q_ia_i<\beta _i+1$, we have 
	$$
	-\frac{\tau _i\left( \beta _i+1 \right)}{q_i}<a_i\tau_i,
	$$
	which is equivalent to 
	\begin{equation}\label{qian 1}
	-\frac{\tau _i\left( \beta _i+1 \right)}{q_i}-\frac{a_i\left( n+1+\beta _i \right)}{q_i}<\frac{a_i\left( n+1+\alpha _i \right)}{p_{i}^{\prime}}.
	\end{equation}
	In addition, from $\alpha _i+1<p_i\left( b_i+1 \right)$, it follows that 
	$$
	b_i-\alpha _i+\frac{\alpha _i+1}{p_{i}^{\prime}}>0.
	$$
	Hence, we obtain 
	\begin{equation}\label{qian 2}
	-\frac{\left( \alpha _i+1 \right) \tau _i}{p_{i}^{\prime}}-\frac{\left( b_i-\alpha _i \right) \left( n+1+\alpha _i \right)}{p_{i}^{\prime}}<\frac{\left( b_i-\alpha _i \right) \left( n+1+\beta _i \right)}{q_i}.
	\end{equation}
	By (\ref{qian 1}) and (\ref{qian 2}), there exist $r_i$ and $s_i$ such that 
	\begin{equation*}
	-\frac{\tau _i\left( \beta _i+1 \right)}{q_i}-\frac{a_i\left( n+1+\beta _i \right)}{q_i}<r_i\tau _i+a_i\left( r_i-s_i \right) <\frac{a_i\left( n+1+\alpha _i \right)}{p_{i}^{\prime}}
	\end{equation*}
	and
	\begin{equation*}
	-\frac{\tau _i\left( \alpha _i+1 \right)}{p_{i}^{\prime}}-\frac{\left( b_i-\alpha _i \right) \left( n+1+\alpha _i \right)}{p_{i}^{\prime}}<s_i\tau _i+\left( b_i-\alpha _i \right) \left( s_i-r_i \right) <\frac{\left( b_i-\alpha _i \right) \left( n+1+\beta _i \right)}{q_i},
	\end{equation*}
	which is equivalent to
	\begin{equation}\label{r de}
	-\frac{ \beta _i+1 }{q_i}-a_i\delta _i<r_i<a_i\gamma _i
	\end{equation}
	and
	\begin{equation}\label{s de}
	-\frac{\alpha _i+1}{p_{i}^{\prime}}-\left( b_i-\alpha _i \right) \gamma _i<s_i<\left( b_i-\alpha _i \right) \delta _i,
	\end{equation}
	where 
	$$
	\gamma _i=\frac{\left( n+1+\alpha _i \right) /p_{i}^{\prime}+s_i-r_i}{\tau _i}
	$$ 
	and 
	$$
	\delta _i=\frac{\left( n+1+\beta _i \right) /q_i+r_i-s_i}{\tau _i}.
	$$
	
	Obviously, for any $i\in\{1,2\}$, $\gamma_i+\delta_i=1$.
	
	Let $h_1\left( u,\eta \right) =\rho \left( u \right) ^{s_1}\rho \left( \eta \right) ^{s_2}$, $h_2\left( z,w \right) =\rho \left( z \right) ^{r_1}\rho \left( w \right) ^{r_2}$, write $S_{\vec{a},\ \vec{b},\ \vec{c}}$ as
	$$
	S_{\vec{a},\ \vec{b},\ \vec{c}}\ f(z,w)=\int_{T_B}{\int_{T_B}{K_1\left( z,u \right) K_2\left( w,\eta \right)}f\left( u,\eta \right) dV_{\alpha _1}\left( u \right) dV_{\alpha _2}\left( \eta \right)},
	$$
	where
	$$
	K_1\left( z,u \right) =\frac{\rho \left( u \right) ^{b_1-\alpha _1}\rho(z)^{a_1}}{\left| \rho \left( z,u \right) \right|^{c_1}},\ \text{and}\ K_2\left( w,\eta \right) =\frac{\rho \left( \eta \right) ^{b_2-\alpha _2}\rho(w)^{a_2}}{\left| \rho \left( w,\eta \right) \right|^{c_2}}.
	$$
	
	Now we prove this conclusion by lemma \ref{S1}. We consider
	\begin{equation}\label{t11}
	\begin{aligned}
	&\int_{T_B}{\left( K_1\left( z,u \right) \right) ^{\gamma _1p_{1}^{\prime}}\left( K_2\left( w,\eta \right) \right) ^{\gamma _2p_{1}^{\prime}}\left( h_1\left( u,\eta \right) \right) ^{p_{1}^{\prime}}dV_{\alpha _1}\left( u \right)}
	\\
	&=\frac{\rho(z)^{a_1\gamma_1p_1^{\prime}}\rho(w)^{a_2\gamma_2p_2^{\prime}}\rho \left( \eta \right) ^{\left( b_2-\alpha _2 \right) \gamma _2p_{1}^{\prime}+s_2p_{1}^{\prime}}}{\left| \rho \left( w,\eta \right) \right|^{c_2\gamma _2p_{1}^{\prime}}}\\
	&\times\int_{T_B}{\frac{\rho \left( u \right) ^{\left( b_1-\alpha _1 \right) \gamma _1p_{1}^{\prime}+s_1p_{1}^{\prime}+\alpha _1}}{\left| \rho \left( z,u \right) \right|^{c_1\gamma _1p_{1}^{\prime}}}dV\left( u \right)}.
	\end{aligned}
	\end{equation}

	From the first inequality in (\ref{s de}), we have 
	\begin{equation}\label{t12}
	\left( b_i-\alpha _i \right) \gamma _ip_{i}^{\prime}+s_ip_{i}^{\prime}+\alpha _i>-1.
	\end{equation}
	Notice that 
	$$
	\left(c_i-a_i-b_i+\alpha _i\right)\gamma_i=\tau _i\gamma_i=\frac{n+1+\alpha _i}{p_{i}^{\prime}}+s_i-r_i,
	$$
	by the second inequality in (\ref{r de}), we have 
	\begin{equation}\label{t13}
	c_i\gamma _ip_{i}^{\prime}-n-1-\left( b_i-\alpha _i \right) \gamma _ip_{i}^{\prime}-s_ip_{i}^{\prime}-\alpha _i=a_i\gamma _ip_{i}^{\prime}-r_ip_{i}^{\prime}>0.
	\end{equation}
	Thus, according to Lemma \ref{lem jifendengshi}, for any given $z\in T_B$, we have
	$$
	\int_{T_B}{\frac{\rho \left( u \right) ^{\left( b_1-\alpha _1 \right) \gamma _1p_{1}^{\prime}+s_1p_{1}^{\prime}+\alpha _1}}{\left| \rho \left( z,u \right) \right|^{c_1\gamma _1p_{1}^{\prime}}}dV\left( u \right)}\lesssim\rho(z)^{r_1p_{1}^{\prime}-a_1\gamma _1p_{1}^{\prime}},
	$$
	which, together with (\ref{t11}), (\ref{t12}), (\ref{t13}), and Lemma \ref{lem jifendengshi}, we get 
	$$
	\begin{aligned}
	\int_{T_B}&{\left( \int_{T_B}{\left( K_1\left( z,u \right) \right) ^{\gamma _1p_{1}^{\prime}}\left( K_2\left( w,\eta \right) \right) ^{\gamma _2p_{1}^{\prime}}\left( h_1\left( u,\eta \right) \right) ^{p_{1}^{\prime}}dV_{\alpha _1}\left( u \right)} \right) ^{\frac{p_{2}^{\prime}}{p_{1}^{\prime}}}dV_{\alpha _2}\left( \eta \right)}
	\\
	&\lesssim \rho \left( z \right) ^{r_1p_{2}^{\prime}}\rho(w)^{a_2\gamma_2p_2^{\prime}}\int_{T_B}{\frac{\rho \left( \eta \right) ^{\left( b_2-\alpha _2 \right) \gamma _2p_{2}^{\prime}+s_2p_{2}^{\prime}+\alpha _2}}{\left| \rho \left( w,\eta \right) \right|^{c_2\gamma _2p_{2}^{\prime}}}dV\left( \eta \right)}
	\\
	&\lesssim \rho \left( z \right) ^{r_1p_{2}^{\prime}}\rho \left( w \right) ^{r_2p_{2}^{\prime}}=h_2\left( z,w \right) ^{p_{2}^{\prime}}.
	\end{aligned}
	$$
	Thus, condition (\ref{3.1}) holds ture for the operator $S_{\vec{a},\ \vec{b},\ \vec{c}}$.
	\\

	Next, we verify condition (\ref{3.2}). Notice that 
	$$
	\begin{aligned}
	&\int_{T_B}{\left( K_1\left( z,u \right) \right) ^{\delta _1q_1}\left( K_2\left( w,\eta \right) \right) ^{\delta _2q_1}\left( h_2\left( z,w \right) \right) ^{q_1}dV_{\beta _1}\left( z \right)}
	\\
	&=\frac{\rho \left( u \right) ^{\left( b_1-\alpha _1 \right) \delta _1q_1}\rho \left( \eta \right) ^{\left( b_2-\alpha _2 \right) \delta _2q_1}\rho \left( w \right) ^{r_2q_1+a_2\delta_2q_1}}{\left| \rho \left( w,\eta \right) \right|^{c_2\delta _2q_1}}\int_{T_B}{\frac{\rho \left( z \right) ^{r_1q_1+a_1\delta_1q_1+\beta _1}}{\left| \rho \left( z,u \right) \right|^{c_1\delta _1q_1}}dV\left( z \right)}.
	\end{aligned}
	$$
	From the first inequality in (\ref{r de}), we have 
	\begin{equation}\label{t22}
	r_iq_i+a_i\delta_iq_i+\beta_i>-1.
	\end{equation}
	Notice that 
	$$
	\left(c_i-a_i-b_i+\alpha _i\right)\delta_i=\tau _i\delta_i=\frac{n+1+\beta _i}{q_{i}}+r_i-s_i,
	$$
	by the second inequality in (\ref{s de}), we have 
	\begin{equation}\label{t23}
	c_i\delta _iq_{i}-n-1-r_iq_i-a_i\delta_iq_i-\beta_i=\left( b_i-\alpha _i \right) \delta _iq_i-s_iq_i>0.
	\end{equation}
	Thus, according to Lemma \ref{lem jifendengshi}, for any given $u\in T_B$, we have
	$$
	\int_{T_B}{\frac{\rho \left( z \right) ^{r_1q_1+a_1\delta_1q_1+\beta _1}}{\left| \rho \left( z,u \right) \right|^{c_1\delta _1q_1}}dV\left( z \right)}\lesssim\rho(u)^{s_1q_1-\left( b_1-\alpha _1 \right) \delta _1q_1},
	$$
	which, together with (\ref{t22}), (\ref{t23}), and Lemma \ref{lem jifendengshi}, we get 
	$$
	\begin{aligned}
	& \int_{T_B}\left[\int_{T_B}\left[K_1(z, u)\right]^{\delta_1 q_1}\left[K_2(w, \eta)\right]^{\delta_2 q_1}\left[h_2(z, w)\right]^{q_1} d V_{\beta_1}(z)\right]^{q_2 / q_1} d V_{\beta_2}(w) \\
	& \quad \lesssim \rho(u)^{s_1 q_2}\rho(\eta)^{\left(b_2-\alpha_2\right) \delta_2 q_2} \int_{T_B} \frac{\rho(w)^{r_2 q_2+a_2\delta_2q_2+\beta_2}}{|\rho(w,\eta)|^{c_2 \delta_2 q_2}} dV(w) \\
	& \quad \lesssim\rho(u)^{s_1 q_2}\rho(\eta)^{s_2 q_2} =\left[h_1(u, \eta)\right]^{q_2} .
	\end{aligned}
	$$
	
	Thus, condition (\ref{3.2}) holds ture for the operator $S_{\vec{a},\ \vec{b},\ \vec{c}}$.
	
	Therefore, the operator $S_{\vec{a}, \ \vec{b},\ \vec{c}}$ satisfies all the conditions of Lemma \ref{S1}, then operator $S_{\vec{a}, \ \vec{b},\ \vec{c}}$ is bounded from $L_{\vec{\alpha}}^{\vec{p}}\left( T_B\times T_B \right)$ to $L_{\vec{\beta}}^{\vec{q}}\left( T_B\times T_B \right)$.
\end{proof}

\begin{lemma}\label{lem 1 2}
	Let $1=p_+\le q_-\le q_+<\infty$. If the parameters satisfy for any $i \in \{1,2\}$, 
	$$
	\begin{cases}
	-q_ia_i<\beta _i+1,\  \alpha _i< b_i ,\\
	c_i=n+1+a_i+b_i+\lambda_i,\\
	\end{cases}
	$$
	\\
	where 
	$$
	\lambda_i =\frac{n+1+\beta _i}{q_i}-n-1-\alpha _i.
	$$
	Then the operator $S_{\vec{a},\ \vec{b},\ \vec{c}}$ is bounded from $L_{\vec{\alpha}}^{\vec{1}}\left( T_B\times T_B \right)$ to $L_{\vec{\beta}}^{\vec{q}}\left( T_B\times T_B \right)$.
\end{lemma}

\begin{proof}
	
	When $p_1=p_2=1$, for any $i\in \{1,2\}$, write $\gamma _i=\frac{s_i-r_i}{\tau _i}$, $\delta _i=\frac{\left( n+1+\beta _i \right) /q_i+r_i-s_i}{\tau _i}$. 
	According to the proof of Lemma \ref{lem 1 1}, we have 
	\begin{equation}\label{r de 2}
	-\frac{ \beta _i+1 }{q_i}-a_i\delta _i<r_i<a_i\gamma _i
	\end{equation}
	
	and
	
	\begin{equation}\label{s de 2}
	-\left( b_i-\alpha _i \right) \gamma _i<s_i<\left( b_i-\alpha _i \right) \delta _i.
	\end{equation}
	
	Next, we use lemma \ref{S2} to prove this lemma. 
	
	First, we consider 
	$$
	\left( K_1\left( z,u \right) \right) ^{\gamma _1}\left( K_2\left( w,\eta \right) \right) ^{\gamma _2}h_1\left( u,\eta \right) =\frac{\rho(z)^{a_1\gamma_1}\rho(w)^{a_2\gamma_2}\rho \left( u \right) ^{\gamma _1\left( b_1-\alpha _1 \right) +s_1}\rho \left( \eta \right) ^{\gamma _2\left( b_2-\alpha _2 \right) +s_2}}{\left| \rho \left( z,u \right) \right|^{c_1\gamma _1}\left| \rho \left( w,\eta \right) \right|^{c_2\gamma _2}}.
	$$
	
	By lemma \ref{lem qu bian liang}, we have for any $z\in T_B$ and $u\in T_B$,
	$$
	2\left| \rho \left( z,u \right) \right|\ge \max \left\{ \rho \left( z \right) ,\rho \left( u \right) \right\} .
	$$
	For any $i\in\{1,2\}$, we get 
	\begin{equation}\label{t31}
	c_i\gamma _i=\left( b_i-\alpha _i +a_i\right) \gamma _i+s_i-r_i
	\end{equation}
	due to $\tau _i=c_i-a_i-b_i+\alpha _i$ and $\gamma _i=\left( s_i-r_i \right) /\tau _i$.
	
	From the first inequality in (\ref{s de 2}), we have
	\begin{equation}\label{t32}
	\gamma _i\left( b_i-\alpha _i \right) +s_i>0.
	\end{equation}
	Then, according to Lemma \ref{lem qu bian liang} and (\ref{t31}), for any given $z\in T_B$ and $u\in T_B$, we have
	$$
	\frac{\rho \left( z \right) ^{a_1\gamma_1}\rho \left( u \right) ^{\left( b_1-\alpha _1 \right)\gamma _1 +s_1}}{\left| \rho \left( z,u \right) \right|^{c_1\gamma _1}}=\left( \frac{\rho \left( u \right)}{\left| \rho \left( z,u \right) \right|} \right) ^{\left( b_1-\alpha _1 \right)\gamma _1 +s_1}\left( \frac{\rho \left( z \right)}{\left| \rho \left( z,u \right) \right|} \right) ^{a_1\gamma_1}\lesssim \rho(z)^{r_1}	
	$$
	and, similarly, for any given  $w\in T_B$ and any $\eta \in T_B$,
	$$
	\frac{\rho \left( w \right) ^{a_2\gamma_2}\rho \left( \eta \right) ^{\left( b_2-\alpha _2 \right)\gamma _2 +s_2}}{\left| \rho \left( w,\eta \right) \right|^{c_2\gamma _2}}=\left( \frac{\rho \left( \eta \right)}{\left| \rho \left( w,\eta \right) \right|} \right) ^{\left( b_2-\alpha _2 \right)\gamma _2 +s_2}\left( \frac{\rho \left( w \right)}{\left| \rho \left( w,\eta \right) \right|} \right) ^{a_2\gamma_2}\lesssim \rho(w)^{r_2}.
	$$
	Thus, for any given $\left( z,w \right) \in T_B\times T_B$,
	$$
	\mathop{\text{ess\ sup}}_{{\left( u,\eta \right)}\in T_B\times T_B}\left( K_1\left( z,u \right) \right) ^{\gamma _1}\left( K_2\left( w,\eta \right) \right) ^{\gamma _2}h_1\left( u,\eta \right) \le Ch_2\left( z,w \right) .
	$$
	Thus, condition (\ref{3.3}) holds true for the operator $S_{\vec{a},\ \vec{b},\ \vec{c}}$. 
	
	Note that the condition (\ref{3.4}) in Lemma \ref{S2} is the same as the condition (\ref{3.2}) in Lemma \ref{S1}, thus according to the proof of the second part of Lemma \ref{lem 1 1}, we know that the condition (\ref{3.4}) in Lemma \ref{S2} still holds.
	
	Therefore, the operator $S_{\vec{a}, \ \vec{b},\ \vec{c}}$ satisfies all the conditions of Lemma \ref{S2}, then operator $S_{\vec{a}, \ \vec{b},\ \vec{c}}$ is bounded from $L_{\vec{\alpha}}^{\vec{1}}\left( T_B\times T_B \right)$ to $L_{\vec{\beta}}^{\vec{q}}\left( T_B\times T_B \right)$.
\end{proof}

\begin{lemma}\label{lem 1 3}
	Let $1=p_2<p_1\le q_-\le q_+<\infty$. If the parameters satisfy that, for any $i\in\{1,2\}$,
	$$
	\begin{cases}
	\alpha _1+1<p_1\left( b_1+1 \right) ,\ \alpha _2<b_2,\\
	-q_ia_i<\beta_i+1,\ c_i=n+1+a_i+b_i+\lambda_i,\\
	\end{cases}
	$$
	where
	$$
	\begin{cases}
	\lambda _1=\frac{n+1+\beta _1}{q_1}-\frac{n+1+\alpha _1}{p_1},\\
	\lambda _2=\frac{n+1+\beta _2}{q_2}-\left( n+1+\alpha _2 \right).\\
	\end{cases}
	$$
	Then the operator $S_{\vec{a},\ \vec{b},\ \vec{c}}$ is bounded from $L_{\vec{\alpha}}^{\vec{p}}\left( T_B\times T_B \right)$ to $L_{\vec{\beta}}^{\vec{q}}\left( T_B\times T_B \right)$.
\end{lemma}

\begin{proof}
	
	For any $i\in \{1,2\}$, let $\tau _i=c_i-a_i-b_i+\alpha _i$, 
	\\
	then 
	$$
	\tau _1=\frac{n+1+\alpha _1}{p_{1}^{\prime}}+\frac{n+1+\beta _1}{q_1}>0
	$$
	and
	$$
	\tau _2=\frac{n+1+\beta _2}{q_2}>0.
	$$
	
	Since $-q_1a_1<\beta _1+1$, we have 
	$$
	-\frac{\tau _1\left( \beta _1+1 \right)}{q_1}<a_1\tau_1,
	$$
	which is equivalent to 
	\begin{equation}\label{qian 3}
	-\frac{\tau _1\left( \beta _1+1 \right)}{q_1}-\frac{a_1\left( n+1+\beta _1 \right)}{q_1}<\frac{a_1\left( n+1+\alpha _1 \right)}{p_{1}^{\prime}}.
	\end{equation}
	In addition, from $\alpha _1+1<p_1\left( b_1+1 \right)$, it follows that 
	$$
	b_1-\alpha _1+\frac{\alpha _1+1}{p_{1}^{\prime}}>0.
	$$
	Combined with the known fact that $b_2>\alpha_2$, further implies that
	\begin{equation}\label{qian 4}
	-\frac{\tau _1\left( 1+\alpha _1 \right)}{p_{1}^{\prime}}-\frac{\left( b_1-\alpha _1 \right) \left( n+1+\alpha _1 \right)}{p_{1}^{\prime}}<\frac{\left( b_1-\alpha _1 \right) \left( n+1+\beta _1 \right)}{q_1}.
	\end{equation}	
	Thus, by (\ref{qian 3}) and (\ref{qian 4}), there exists $r_1$ and $s_1$ such that 
	\begin{equation*}
	-\frac{\tau _1\left( \beta _1+1 \right)}{q_1}-\frac{a_1\left( n+1+\beta _1\right)}{q_1}<r_1\tau _1+a_1\left( r_1-s_1 \right) <\frac{a_1\left( n+1+\alpha _1 \right)}{p_{1}^{\prime}}
	\end{equation*}
	and	
	$$
	\frac{\tau _1\left( 1+\alpha _1 \right)}{p_{1}^{\prime}}-\frac{\left( b_1-\alpha _1 \right) \left( n+1+\alpha _1 \right)}{p_{1}^{\prime}}
	<\tau _1s_1+\left( b_1-\alpha _1 \right) \left( s_1-r_1 \right) <\frac{\left( b_1-\alpha _1 \right) \left( n+1+\beta _1 \right)}{q_1}.
	$$
	
	Since $-q_2a_2<\beta _2+1$, we have 
	$$
	-\frac{\tau _2\left( \beta _2+1 \right)}{q_2}<a_2\tau_2,
	$$
	which is equivalent to 
	\begin{equation}\label{qian 31}
	-\frac{\tau _2\left( \beta _2+1 \right)}{q_2}-\frac{a_2\left( n+1+\beta _2 \right)}{q_2}<0.
	\end{equation}
	In addition, from $\alpha _2< b_2$, it follows that 
	\begin{equation}\label{qian 41}
	\frac{\left( b_2-\alpha _2 \right) \left( n+1+\beta _2 \right)}{q_2}>0.
	\end{equation}
	Thus, by (\ref{qian 31}) and (\ref{qian 41}), there exists $r_2$ and $s_2$ such that 
	$$
	-\frac{\tau _2\left( \beta _2+1 \right)}{q_2}-\frac{a_2\left( n+1+\beta _2\right)}{q_2}<r_2\tau _2+a_2\left( r_2-s_2 \right) <0
	$$
	and
	$$
	0<\tau _2s_2+\left( b_2-\alpha _2 \right) \left( s_2-r_2 \right) <\frac{\left( b_2-\alpha _2 \right) \left( n+1+\beta _2 \right)}{q_2}.
	$$
	That is 
	\begin{equation}\label{r de 1}
	-\frac{ \beta _i+1 }{q_i}-a_i\delta _i<r_i<a_i\gamma _i,
	\end{equation}

	\begin{equation}\label{S1 de}
	-\frac{1+\alpha _1}{p_{1}^{\prime}}-\left( b_1-\alpha _1 \right) \gamma _1<s_1<\left( b_1-\alpha _1 \right) \delta _1
	\end{equation}
	and
	\begin{equation}\label{S2 de}
	-\left( b_2-\alpha _2 \right) \gamma _2<s_2<\left( b_2-\alpha _2 \right) \delta _2,
	\end{equation}

	where 
	$$
	\gamma _1=\frac{\left( n+1+\alpha _1 \right) /p_{1}^{\prime}+s_1-r_1}{\tau _1},
	$$
	$$
	\delta _1=\frac{\left( n+1+\beta _1 \right) /q_1+r_1-s_1}{\tau _1},
	$$
	$$
	\gamma _2=\frac{s_2-r_2}{\tau _2}
	$$
	and
	$$
	\delta _2=\frac{\left( n+1+\beta _2 \right) /q_2+r_2-s_2}{\tau _2}.
	$$
	Clearly, $\gamma _1+\delta _1=1$, $\gamma _2+\delta _2=1$.
	
	Now, let $h_1\left( u,\eta \right) =\rho \left( u \right) ^{s_1}\rho \left( \eta \right) ^{s_2}$, $h_2\left( z,w \right) =\rho \left( z \right) ^{r_1}\rho \left( w \right) ^{r_2}$, write $S_{\vec{a},\ \vec{b},\ \vec{c}}$ as
	$$
	S_{\vec{a},\ \vec{b},\ \vec{c}}\ f(z,w)=\int_{T_B}{\int_{T_B}{K_1\left( z,u \right) K_2\left( w,\eta \right)}f\left( u,\eta \right) dV_{\alpha _1}\left( u \right) dV_{\alpha _2}\left( \eta \right)},
	$$
	where
	$$
	K_1\left( z,u \right) =\frac{\rho \left( u \right) ^{b_1-\alpha _1}\rho(z)^{a_1}}{\left| \rho \left( z,u \right) \right|^{c_1}}\ \text{and}\ K_2\left( w,\eta \right) =\frac{\rho \left( \eta \right) ^{b_2-\alpha _2}\rho(w)^{a_2}}{\left| \rho \left( w,\eta \right) \right|^{c_2}}.
	$$
	
	Next we prove this result by Lemma \ref{S3}. 
	
	We consider	
	\begin{equation}\label{t41}
	\begin{aligned}
	&\int_{T_B}{\left( K_1\left( z,u \right) \right) ^{\gamma _1p_{1}^{\prime}}\left( K_2\left( w,\eta \right) \right) ^{\gamma _2p_{1}^{\prime}}\left( h_1\left( u,\eta \right) \right) ^{p_{1}^{\prime}}dV_{\alpha _1}\left( u \right)}
	\\
	&=\frac{\rho(z)^{a_1\gamma_1p_1^{\prime}}\rho(w)^{a_2\gamma_2p_2^{\prime}}\rho \left( \eta \right) ^{\left( b_2-\alpha _2 \right) \gamma _2p_{1}^{\prime}+s_2p_{1}^{\prime}}}{\left| \rho \left( w,\eta \right) \right|^{c_2\gamma _2p_{1}^{\prime}}}\\
	&\times\int_{T_B}{\frac{\rho \left( u \right) ^{\left( b_1-\alpha _1 \right) \gamma _1p_{1}^{\prime}+s_1p_{1}^{\prime}+\alpha _1}}{\left| \rho \left( z,u \right) \right|^{c_1\gamma _1p_{1}^{\prime}}}dV\left( u \right)}.
	\end{aligned}
	\end{equation}
	
	From the first inequality in (\ref{S1 de}), we have 
	\begin{equation}\label{t33}
	\left( b_1-\alpha _1 \right) \gamma _1p_{1}^{\prime}+s_1p_{1}^{\prime}+\alpha _1>-1.
	\end{equation}
	Notice that 
	$$
	\left(c_1-a_1-b_1+\alpha _1\right)\gamma_1=\tau _1\gamma_1=\frac{n+1+\alpha _1}{p_{1}^{\prime}}+s_1-r_1,
	$$
	by the second inequality in (\ref{r de 1}), we have 
	\begin{equation}\label{t34}
	c_1\gamma _1p_{1}^{\prime}-n-1-\left( b_1-\alpha _1 \right) \gamma _1p_{1}^{\prime}-s_1p_{1}^{\prime}-\alpha _1=a_1\gamma _1p_{1}^{\prime}-r_1p_{1}^{\prime}>0.
	\end{equation}
	In addition, from Lemma \ref{lem qu bian liang} and the fact $c_2\gamma_2+r_2-a_2\gamma_2=\left(b_2-\alpha_2\right)\gamma_2+s_2>0$, we infer that, for any $w\in T_B$ and $\eta \in T_B$,
	\begin{equation}\label{t35}
	\frac{\rho \left( \eta \right) ^{\left( b_2-\alpha _2 \right) \gamma _2p_{1}^{\prime}+s_2p_{1}^{\prime}}}{\rho \left( w,\eta \right) ^{c_2\gamma _2p_{1}^{\prime}}}\lesssim \rho \left( w \right) ^{\left(r_2-a_2\gamma_2\right)p_{1}^{\prime}}.
	\end{equation}
	Thus, according to Lemma \ref{lem jifendengshi}, for any given $z\in T_B$, we have
	\begin{equation}\label{t36}
	\int_{T_B}{\frac{\rho \left( u \right) ^{\left( b_1-\alpha _1 \right) \gamma _1p_{1}^{\prime}+s_1p_{1}^{\prime}+\alpha _1}}{\left| \rho \left( z,u \right) \right|^{c_1\gamma _1p_{1}^{\prime}}}dV\left( u \right)}\lesssim \rho \left( z \right) ^{r_1p_{1}^{\prime}-a_1\gamma_1p_{1}^{\prime}},
	\end{equation}
	combining  with (\ref{t35}) and (\ref{t36}), we further obtain that
	$$
	\begin{aligned}
	\mathop {\mathrm{ess}\ \mathrm{sup}} \limits_{{\eta \in T_B}}\int_{T_B}&{\left( K_1\left( z,u \right) \right) ^{\gamma _1p_{1}^{\prime}}\left( K_2\left( w,\eta \right) \right) ^{\gamma _2p_{1}^{\prime}}\left( h_1\left( u,\eta \right) \right) ^{p_{1}^{\prime}}dV_{\alpha _1}\left( u \right)}
	\\
	&\lesssim \rho \left( z \right) ^{r_1p_{1}^{\prime}}\rho \left( w \right) ^{r_2p_{1}^{\prime}}=\left( h_2\left( z,w \right) \right) ^{p_{1}^{\prime}}.
	\end{aligned}
	$$
	Thus, condition (\ref{3.5}) holds ture for the operator $S_{\vec{a},\ \vec{b},\ \vec{c}}$.
	
	Next, we verify condition (\ref{3.6}) of Lemma \ref{S3}. Notice that 
	$$
	\begin{aligned}
	&\int_{T_B}{\left( K_1\left( z,u \right) \right) ^{\delta _1q_1}\left( K_2\left( w,\eta \right) \right) ^{\delta _2q_1}\left( h_2\left( z,w \right) \right) ^{q_1}dV_{\beta _1}\left( z \right)}
	\\
	&=\frac{\rho \left( u \right) ^{\left( b_1-\alpha _1 \right) \delta _1q_1}\rho \left( \eta \right) ^{\left( b_2-\alpha _2 \right) \delta _2q_1}\rho \left( w \right) ^{r_2q_1+a_2\delta_2q_1}}{\left| \rho \left( w,\eta \right) \right|^{c_2\delta _2q_1}}\int_{T_B}{\frac{\rho \left( z \right) ^{r_1q_1+a_1\delta_1q_1+\beta _1}}{\left| \rho \left( z,u \right) \right|^{c_1\delta _1q_1}}dV\left( z \right)}.
	\end{aligned}
	$$
	
	From the first inequality in (\ref{r de 1}), we obviously have
	\begin{equation}\label{t37}
	r_i q_i+a_i \delta_i q_i+ \beta _i >-1.
	\end{equation}

	Notice that 
	$$
	\left(c_i-a_i-b_i+\alpha _i\right)\delta_i=\tau _i\delta_i=\frac{n+1+\beta _i}{q_i}+r_i-s_i,
	$$
	by the second inequality in (\ref{S1 de}) and (\ref{S2 de}), we have 
	\begin{equation}\label{t38}
	c_i\delta_iq_i-n-1-a_i\delta_iq_i-r_iq_i-\beta _i=\left(b_i-\alpha_i\right)\delta_iq_i-s_iq_i>0.
	\end{equation}
	Thus, according to Lemma \ref{lem jifendengshi}, for any given $u\in T_B$, we have
	\begin{equation}\label{t39}
	\int_{T_B}{\frac{\rho \left( z \right) ^{r_1q_1+a_1\delta_1q_1+\beta_1}}{\left| \rho \left( z,u \right) \right|^{c_1\delta _1q_1}}dV\left( u \right)}\lesssim\rho(u)^{s_1q_1-\left(b_1-\alpha_1\right)\delta_1q_1},
	\end{equation}	
	combining with (\ref{t37}), (\ref{t38}) and Lemma \ref{lem jifendengshi}, this imply that 
	$$
	\begin{aligned}
	\int_{T_B}&{\left( \int_{T_B}{\left( K_1(z,u) \right) ^{\delta _1q_1}}\left( K_2(w,\eta ) \right) ^{\delta _2q_1}\left( h_2(z,w) \right) ^{q_1}dV_{\beta _1}(z) \right) ^{q_2/q_1}}dV_{\beta _2}(w)
	\\
	&\lesssim \rho \left( \eta \right) ^{\left( b_2-\alpha _2 \right) \delta _2q_2}\rho \left( u \right) ^{s_1q_2}\int_{T_B}{\frac{\rho \left( w \right) ^{r_2q_2+a_2\delta_2q_2+\beta _2}}{\left| \rho \left( w,\eta \right) \right|^{c_2\delta _2q_2}}dV\left( w \right)}\\
	&\lesssim\rho(u)^{s_1q_2}\rho(\eta)^{s_2q_2}.
	\end{aligned}
	$$
	
	Thus, condition (\ref{3.6}) holds ture for the operator $S_{\vec{a},\ \vec{b},\ \vec{c}}$.
	
	Then the operator $S_{\vec{a}, \ \vec{b},\ \vec{c}}$ satisfies all the conditions of Lemma \ref{S3}, therefore,  operator $S_{\vec{a}, \ \vec{b},\ \vec{c}}$ is bounded from $L_{\vec{\alpha}}^{\vec{p}}\left( T_B\times T_B \right)$ to $L_{\vec{\beta}}^{\vec{q}}\left( T_B\times T_B \right)$.
\end{proof}

\begin{lemma}\label{lem 1 4}
	Let $1=p_1<p_2\le q_-\le q_+<\infty$.  If the parameters satisfy that, for any $i\in\{1,2\}$,
	$$
	\begin{cases}
	\alpha_1<b_1,\  (\alpha _2+1)<p_2(b_2+1),\\
	-q_ia_i<\beta_i+1,\ c_i=n+1+a_i+b_i+\lambda_i,\\
	\end{cases}
	$$
	where
	$$
	\begin{cases}
	\lambda _1=\frac{n+1+\beta _1}{q_1}-\left( n+1+\alpha _1 \right),\\
	\lambda _2=\frac{n+1+\beta _2}{q_2}-\frac{n+1+\alpha _2}{p_2}.\\
	\end{cases}
	$$	
	Then the operator $S_{\vec{a},\ \vec{b},\ \vec{c}}$ is bounded from $L_{\vec{\alpha}}^{\vec{p}}\left( T_B\times T_B \right)$ to $L_{\vec{\beta}}^{\vec{q}}\left( T_B\times T_B \right)$.
\end{lemma}

\begin{proof}
	
	Lemma \ref{lem 1 4} is the symmetric case of Lemma \ref{lem 1 3}. We only need to modify the definitions of $\gamma_1$, $\delta_1$, $\gamma_2$ and $\delta_2$, thus we omit the proof here.	
\end{proof}

\begin{lemma}\label{lem 1 5}
	Let $1<p_1<q_1=\infty$ and $1<p_2<q_2=\infty$. If the parameters satisfy that, for any $i\in\{1,2\}$,
	$$
	\begin{cases}
	a_i>0,\ \alpha _i+1<p_i(b_i+1),\\
	c_i=n+1+a_i+b_i+\lambda_i,\\
	\end{cases}
	$$
	where
	$$
	\lambda _i=-\frac{n+1+\alpha _i}{p_i}.
	$$	
	Then the operator $S_{\vec{a},\ \vec{b},\ \vec{c}}$ is bounded from $L_{\vec{\alpha}}^{\vec{p}}\left( T_B\times T_B \right)$ to $L^{\infty}\left( T_B\times T_B \right)$.
\end{lemma}

\begin{proof}
	According to the definition of operator $S_{\vec{a},\ \vec{b},\ \vec{c}}$, we know that the kernel of the integral is $$
	K_1(z,u)\cdot K_2(w,\eta)=\frac{\rho(z)^{a_1}\rho(u)^{b_1-\alpha_1}\rho(w)^{a_2}\rho(\eta)^{b_2-\alpha_2}}{\left| \rho \left( z,u \right) \right|^{c_1}\left| \rho \left( w,\eta \right) \right|^{c_2}}.
	$$	
	By Lemma \ref{ dao wu qiong }, it suffices to prove 
	$$
	\mathop {\mathrm{ess}\ \mathrm{sup}}_{\left( z,w \right) \in T_B\times T_B}\left\| K_1\left( z,\cdot \right) \cdot K_2\left( w,\cdot \right) \right\| _{L_{\vec{\alpha}}^{\vec{p}\prime}}<\infty .
	$$
	By calculation, we have
	$$
	\begin{aligned}
	&\mathop {\mathrm{ess}\ \mathrm{sup}}_{\left( z,w \right) \in T_B\times T_B}\left\| K_1\left( z,\cdot \right) \cdot K_2\left( w,\cdot \right) \right\| _{L_{\vec{\mu}}^{\vec{p}\prime}}\\
	&=\mathop {\mathrm{ess}\ \mathrm{sup}}_{\left( z,w \right) \in T_B\times T_B}
	\rho(z)^{a_1}\left( \int_{T_B}{\frac{\rho \left( u \right) ^{p_{1}^{\prime}\left( b_1-\alpha _1 \right) +\alpha _1}}{\left| \rho \left( z,u \right) \right|^{c_1p_{1}^{\prime}}}dV\left( u \right)} \right) ^{\frac{1}{p_{1}^{\prime}}}
	\\
	&\times \rho(w)^{a_2}\left( \int_{T_B}{\frac{\rho \left( \eta \right) ^{p_{2}^{\prime}\left( b_2-\alpha _2 \right) +\alpha _2}}{\left| \rho \left( w,\eta \right) \right|^{c_2p_{2}^{\prime}}}dV\left( \eta \right)} \right) ^{\frac{1}{p_{2}^{\prime}}}.
	\end{aligned}
	$$
	Combining the known conditions and Lemma \ref{lem jifendengshi}, we can easily conclude that 
	$$
	\left\| K_1\left( z,\cdot \right) \cdot K_2\left( w,\cdot \right) \right\| _{L_{\vec{\alpha}}^{\vec{p}\prime}}
	$$
	is uniformly bounded with respect to $z$ and $w$. 
	
	Therefore, the operator $S_{\vec{a},\ \vec{b},\ \vec{c}}$ is bounded from $L_{\vec{\alpha}}^{\vec{p}}\left( T_B\times T_B \right) $ to $L^\infty\left( T_B\times T_B \right)$.	
\end{proof}

\begin{lemma}\label{lem 1 6}
	Let $1=p_1<q_1=\infty$ and $1=p_2<q_2=\infty$. If the parameters satisfy that, for any $i\in\{1,2\}$,
	$$
	\begin{cases}
	a_i\ge0,\ b_i\ge\alpha _i,\\
	c_i=n+1+a_i+b_i+\lambda_i,\\
	\end{cases}
	$$
	where
	$$
	\lambda _i=-n-1-\alpha _i.
	$$
	Then the operator $S_{\vec{a},\ \vec{b},\ \vec{c}}$ is bounded from $L_{\vec{\alpha}}^{\vec{1}}\left( T_B\times T_B \right)$ to $L^{\infty}\left( T_B\times T_B \right)$.
\end{lemma}

\begin{proof}
	By Lemma \ref{ dao wu qiong }, it suffices to prove 
	$$
	\mathop {\mathrm{ess}\ \mathrm{sup}}_{\left( z,w \right) \in T_B\times T_B}\left\| K_1\left( z,\cdot \right) \cdot K_2\left( w,\cdot \right) \right\| _{\infty}<\infty .
	$$
	By calculation, we have
	$$
	\begin{aligned}
	&\mathop{\text{ess}\ \text{sup}}_{\left( z,w \right) \in T_B\times T_B}\lVert K_1\left( z,\cdot \right) \cdot K_2\left( w,\cdot \right) \rVert _{\infty}
	\\
	&=\mathop{\text{ess}\ \text{sup}}_{\left( z,w \right) \in T_B\times T_B}\mathop{\text{ess}\ \text{sup}}_{\left( u,\eta \right) \in T_B\times T_B}\rho \left( z \right) ^{a_1}\frac{\rho \left( u \right) ^{b_1-\alpha _1}}{\left| \rho \left( z,u \right) \right|^{c_1}}\rho \left( w \right) ^{a_2}\frac{\rho \left( \eta \right) ^{b_2-\alpha _2}}{\left| \rho \left( w,\eta \right) \right|^{c_2}}.
	\end{aligned}
	$$
	Combined with the known conditions, by lemma \ref{lem qu bian liang}, we can easily conclude that 
	$$
	\left\| K_1\left( z,\cdot \right) \cdot K_2\left( w,\cdot \right) \right\| _{\infty}
	$$
	is uniformly bounded with respect to $z$ and $w$. 
	
	Therefore, the operator $S_{\vec{a},\ \vec{b},\ \vec{c}}$ is bounded from $L_{\vec{\alpha}}^{\vec{1}}\left( T_B\times T_B \right) $ to $L^\infty\left( T_B\times T_B \right)$.
\end{proof}

\begin{lemma}\label{lem 1 7}
	Let $1=p_1<q_1=\infty$ and $1<p_2<q_2=\infty$. If the parameters satisfy that, for any $i\in\{1,2\}$,
	$$
	\begin{cases}
	a_1\ge0,\ a_2>0,\\
	b_1\ge \alpha_1,\ p_2\left( b_2+1 \right) >\alpha _2+1,\\
	c_i=n+1+a_i+b_i+\lambda_i,\\
	\end{cases}
	$$
	where
	$$
	\begin{cases}
	\lambda _1=-n-1-\alpha _1,\\
	\lambda_2=-\frac{n+1+\alpha _2}{p_2}.\\
	\end{cases}
	$$	
	Then the operator $S_{\vec{a},\ \vec{b},\ \vec{c}}$ is bounded from $L_{\vec{\alpha}}^{\vec{p}}\left( T_B\times T_B \right)$ to $L^{\infty}\left( T_B\times T_B \right)$.	
\end{lemma}

\begin{proof}
	This lemma is a direct corollary of Lemma \ref{lem 1 5}, Lemma \ref{lem 1 6} and Lemma \ref{ dao wu qiong }, which we omit to prove.
\end{proof}

\begin{lemma}\label{lem 1 8}
	Let $1<p_1<q_1=\infty$ and $1=p_2<q_2=\infty$. If the parameters satisfy that, for any $i\in\{1,2\}$,
	$$
	\begin{cases}
	a_1>0,\ a_2\ge0, \\
	p_1\left( b_1+1 \right) >\alpha _1+1,\ b_2\ge \alpha_2,\\
	 c_i=n+1+a_i+b_i+\lambda_i,\\
	\end{cases}
	$$
	where
	$$
	\begin{cases}
	\lambda_1=-\frac{n+1+\alpha _1}{p_1},\\
	\lambda _2=-n-1-\alpha _2.\\
	\end{cases}
	$$	
	Then the operator $S_{\vec{a},\ \vec{b},\ \vec{c}}$ is bounded from $L_{\vec{\alpha}}^{\vec{p}}\left( T_B\times T_B \right)$ to $L^{\infty}\left( T_B\times T_B \right)$.	
\end{lemma}

\begin{proof}
	Lemma \ref{lem 1 8} is the symmetric case of Lemma \ref{lem 1 7}. Thus, we omit the proof here.	
\end{proof}

\begin{lemma}\label{lem 1 9}
	Let $p_1=q_1=\infty$ and $p_2=q_2=\infty$. If the parameters satisfy that, for any $i\in\{1,2\}$,
	\begin{equation}\label{t 1 11}
	\begin{cases}
	a_i>0,\ b_i>-1,\\
	c_i=n+1+a_i+b_i.\\
	\end{cases}
	\end{equation}
	Then the operator $S_{\vec{a},\ \vec{b},\ \vec{c}}$ is bounded from $L^{\infty}\left( T_B\times T_B \right)$ to $L^{\infty}\left( T_B\times T_B \right)$.
\end{lemma}

\begin{proof}
	Obviously, when the parameter satisfies condition (\ref{t 1 11}), 
	$$
	\left| S_{\vec{a},\ \vec{b},\ \vec{c}}\ f\left( z,w \right) \right|\lesssim \rho \left( z \right) ^{a_1}\int_{T_B}{\frac{\rho \left( u \right) ^{b_1}}{\left| \rho \left( z,u \right) \right|^{c_1}}dV\left( u \right)}\cdot \rho \left( w \right) ^{a_2}\int_{T_B}{\frac{\rho \left( \eta \right) ^{b_2}}{\left| \rho \left( w,\eta \right) \right|^{c_2}}dV\left( \eta \right)}<\infty.
	$$
	
	Thus, by Lemma \ref{lem jifendengshi}, $S_{\vec{a},\ \vec{b},\ \vec{c}}$ is bounded from $L^{\infty}\left( T_B\times T_B \right) $ to $L^{\infty}\left( T_B\times T_B \right)$.
\end{proof}

\begin{lemma}\label{lem 1 10}
	Let $1=p_1<q_1=\infty$ and $p_2=q_2=\infty$. If the parameters satisfy that, for any $i\in\{1,2\}$,
	$$
	\begin{cases}
	a_1\ge0,\ a_2>0, \\
	 b_1 \ge \alpha _1,\ b_2>-1,\\
	c_i=n+1+a_i+b_i+\lambda_i,\\
	\end{cases}
	$$
	where
	$$
	\begin{cases}
	\lambda _1=-n-1-\alpha _1,\\
	\lambda_2=0.\\
	\end{cases}
	$$	
	Then the operator $S_{\vec{a},\ \vec{b},\ \vec{c}}$ is bounded from $L_{\vec{\alpha}}^{\vec{p}}\left( T_B\times T_B \right)$ to $L^{\infty}\left( T_B\times T_B \right)$.	
	
\end{lemma}

\begin{proof}
	The proof of this lemma is similar to Lemma \ref{lem b10}. Thus, we omit the proof.
\end{proof}

\begin{lemma}\label{lem 1 11}
	Let $p_1=q_1=\infty$ and $1=p_2<q_2=\infty$. If the parameters satisfy that, for any $i\in\{1,2\}$,
	$$
	\begin{cases}
	a_1>0,\ a_2\ge 0,\\
	b_1>-1,\  b_2\ge \alpha_2,\\
	c_i=n+1+a_i+b_i+\lambda_i,\\
	\end{cases}
	$$
	where
	$$
	\begin{cases}
	\lambda_1=0,\\
	\lambda _2=-n-1-\alpha _2.\\
	\end{cases}
	$$
	Then $S_{\vec{a},\ \vec{b},\ \vec{c}}$ is bounded from $L_{\vec{\alpha}}^{\vec{p}}\left( T_B\times T_B \right) $ to $L^{\infty}\left( T_B\times T_B \right)$.
	
\end{lemma}

\begin{proof}
	Lemma \ref{lem 1 10} is the symmetric case of Lemma \ref{lem 1 9}. Thus, we omit the proof here.
\end{proof}

\begin{lemma}\label{lem 1 12}
	Let $p_1=q_1=\infty$ and $1<p_2<q_2=\infty$. If the parameters satisfy that, for any $i\in\{1,2\}$,
	$$
	\begin{cases}
	a_1>0,\ a_2>0,\\
	b_1>-1,\ p_2\left( b_2+1 \right) >\alpha _2+1,\\
	c_i=n+1+a_i+b_i+ \lambda _i ,\\
	\end{cases}	
	$$
	where
	$$
	\begin{cases}
	\lambda _1=0,\\
	\lambda_2=-\frac{n+1+\alpha _2}{p_2}.\\
	\end{cases}
	$$
	Then the operator $S_{\vec{a},\ \vec{b},\ \vec{c}}$ is bounded from $L^{\infty}\left( T_B\times T_B \right)$ to $L_{\vec{\beta}}^{\vec{q}}\left( T_B\times T_B \right)$.
\end{lemma}

\begin{proof}
	The proof of this lemma is similar to Lemma \ref{lem b12}. Thus, we omit the proof.	
\end{proof}

\begin{lemma}\label{lem 1 13}
	Let $1<p_1<q_1=\infty$ and $p_2=q_2=\infty$. If the parameters satisfy that, for any $i\in\{1,2\}$,
	$$
	\begin{cases}
	a_1>0,\ a_2>0,\\
	p_1\left( b_1+1 \right) >\alpha _1+1,\ b_2>-1,\\
	c_i=n+1+a_i+b_i+\lambda_i,\\
	\end{cases}	
	$$
	where
	$$
	\begin{cases}
	\lambda_1=-\frac{n+1+\alpha _2}{p_2},\\
	\lambda _2=0.\\
	\end{cases}
	$$
	Then $S_{\vec{a},\ \vec{b},\ \vec{c}}$ is bounded from $L_{\vec{\alpha}}^{\vec{p}}\left( T_B\times T_B \right) $ to $L^{\infty}\left( T_B\times T_B \right)$. 
\end{lemma}

\begin{proof}
	Lemma \ref{lem 1 13} is the symmetric case of Lemma \ref{lem 1 12}. Thus, we omit the proof here.
\end{proof}

\section{The Proof of Main Theorems}
\ \ \ \
In this section, we will list all the main theorems of this paper and prove them.
\begin{theorem}\label{th 1}
	Let $\vec{p}:=\left( p_1,p_2 \right)$ and $\vec{q}:=\left( q_1,q_2 \right) $ satisfy $1<p_-\le p_+\le q_-\le q_+ <\infty$. Then the following conclusions are equivalent.
	\\
	
	$(1)$ The operator $S_{\vec{a},\ \vec{b},\ \vec{c}}$ is bounded from $L_{\vec{\alpha}}^{\vec{p}}\left( T_B\times T_B \right) $ to $L_{\vec{\beta}}^{\vec{q}}\left( T_B\times T_B \right) $.
	\\
	
	$(2)$ The operator $T_{\vec{a},\ \vec{b},\ \vec{c}}$ is bounded from $L_{\vec{\alpha}}^{\vec{p}}\left( T_B\times T_B \right) $ to $L_{\vec{\beta}}^{\vec{q}}\left( T_B\times T_B \right) $.
	\\
	
	$(3)$ The parameters satisfy that, for any $i\in \{1,2\}$,
	$$
	\begin{cases}
	-q_ia_i<\beta _i+1,\  \alpha _i+1<p_i\left( b_i+1 \right) ,\\
	c_i=n+1+a_i+b_i+\lambda_i,\\
	\end{cases}
	$$
	\\
	where 
	$$
	\lambda_i =\frac{n+1+\beta _i}{q_i}-\frac{n+1+\alpha _i}{p_i}.
	$$
\end{theorem}

\begin{proof}
	$(1)\Rightarrow(2)$ is trivial; $(2)\Rightarrow(3)$ comes from Lemma \ref{lem b1}; $(3)\Rightarrow(1)$ is derived from Lemma \ref{lem 1 1}.	
\end{proof}

\begin{theorem}\label{th 2}
	Let $\vec{p}:=\left(1,1 \right)$ and $\vec{q}:=\left( q_1,q_2 \right) \in [1,\infty)\times[1,\infty)$. Then the following conclusions are equivalent.
	\\
	
	$(1)$ The operator $S_{\vec{a},\ \vec{b},\ \vec{c}}$ is bounded from $L_{\vec{\alpha}}^{\vec{p}}\left( T_B\times T_B \right) $ to $L_{\vec{\beta}}^{\vec{q}}\left( T_B\times T_B \right) $.
	\\
	
	$(2)$ The operator $T_{\vec{a},\ \vec{b},\ \vec{c}}$ is bounded from $L_{\vec{\alpha}}^{\vec{p}}\left( T_B\times T_B \right) $ to $L_{\vec{\beta}}^{\vec{q}}\left( T_B\times T_B \right) $.
	\\
	
	$(3)$ The parameters satisfy that, for any $i\in \{1,2\}$,
	$$
	\begin{cases}
	-q_ia_i<\beta _i+1,\  \alpha _i< b_i ,\\
	c_i=n+1+a_i+b_i+\lambda_i,\\
	\end{cases}
	$$
	\\
	where 
	$$
	\lambda_i =\frac{n+1+\beta _i}{q_i}-n-1-\alpha _i.
	$$
\end{theorem}

\begin{proof}
	$(1)\Rightarrow(2)$ is trivial; $(2)\Rightarrow(3)$ comes from Lemma \ref{lem b2}; $(3)\Rightarrow(1)$ is derived from Lemma \ref{lem 1 2}.	
\end{proof}

\begin{theorem}\label{th 3}
	Let $\vec{p}:=\left( p_1,1 \right)$ and $\vec{q}:=\left( q_1,q_2 \right) $ satisfy $1<p_1\le  q_- \le q_+<\infty$. Then the following conclusions are equivalent.
	\\
	
	$(1)$ The operator $S_{\vec{a},\ \vec{b},\ \vec{c}}$ is bounded from $L_{\vec{\alpha}}^{\vec{p}}\left( T_B\times T_B \right) $ to $L_{\vec{\beta}}^{\vec{q}}\left( T_B\times T_B \right) $.
	\\
	
	$(2)$ The operator $T_{\vec{a},\ \vec{b},\ \vec{c}}$ is bounded from $L_{\vec{\alpha}}^{\vec{p}}\left( T_B\times T_B \right) $ to $L_{\vec{\beta}}^{\vec{q}}\left( T_B\times T_B \right) $.
	\\
	
	$(3)$ The parameters satisfy that, for any $i\in \{1,2\}$,
	$$
	\begin{cases}
	\alpha _1+1<p_1\left( b_1+1 \right) ,\ \alpha _2<b_2,\\
	-q_ia_i<\beta_i+1,\ c_i=n+1+a_i+b_i+\lambda_i,\\
	\end{cases}
	$$
	where
	$$
	\begin{cases}
	\lambda _1=\frac{n+1+\beta _1}{q_1}-\frac{n+1+\alpha _1}{p_1},\\
	\lambda _2=\frac{n+1+\beta _2}{q_2}-\left( n+1+\alpha _2 \right).\\
	\end{cases}
	$$
\end{theorem}

\begin{proof}
	$(1)\Rightarrow(2)$ is trivial; $(2)\Rightarrow(3)$ comes from Lemma \ref{lem b3}; $(3)\Rightarrow(1)$ is derived from Lemma \ref{lem 1 3}.	
\end{proof}

\begin{theorem}\label{th 4}
	Let $\vec{p}:=\left( 1,p_2 \right)$ and $\vec{q}:=\left( q_1,q_2 \right) $ satisfy $1<p_2\le  q_- \le q_+<\infty$. Then the following conclusions are equivalent.
	\\
	
	$(1)$ The operator $S_{\vec{a},\ \vec{b},\ \vec{c}}$ is bounded from $L_{\vec{\alpha}}^{\vec{p}}\left( T_B\times T_B \right) $ to $L_{\vec{\beta}}^{\vec{q}}\left( T_B\times T_B \right) $.
	\\
	
	$(2)$ The operator $T_{\vec{a},\ \vec{b},\ \vec{c}}$ is bounded from $L_{\vec{\alpha}}^{\vec{p}}\left( T_B\times T_B \right) $ to $L_{\vec{\beta}}^{\vec{q}}\left( T_B\times T_B \right) $.
	\\
	
	$(3)$ The parameters satisfy that, for any $i\in \{1,2\}$,
	$$
	\begin{cases}
	\alpha_1<b_1,\  (\alpha _2+1)<p_2(b_2+1),\\
	-q_ia_i<\beta_i+1,\ c_i=n+1+a_i+b_i+\lambda_i,\\
	\end{cases}
	$$
	where
	$$
	\begin{cases}
	\lambda _1=\frac{n+1+\beta _1}{q_1}-\left( n+1+\alpha _1 \right),\\
	\lambda _2=\frac{n+1+\beta _2}{q_2}-\frac{n+1+\alpha _2}{p_2}.\\
	\end{cases}
	$$
\end{theorem}

\begin{proof}
	$(1)\Rightarrow(2)$ is trivial; $(2)\Rightarrow(3)$ comes from Lemma \ref{lem b4}; $(3)\Rightarrow(1)$ is derived from Lemma \ref{lem 1 4}.	
\end{proof}

\begin{theorem}\label{th 5}
	Let $\vec{p}:=\left( p_1,p_2 \right)$ and $\vec{q}:=\left( \infty,\infty \right) $ satisfy $1<p_-\le p_+ <\infty$. Then the following conclusions are equivalent.
	\\
	
	$(1)$ The operator $S_{\vec{a},\ \vec{b},\ \vec{c}}$ is bounded from $L_{\vec{\alpha}}^{\vec{p}}\left( T_B\times T_B \right) $ to $L^{\infty}\left( T_B\times T_B \right)$.
	\\
	
	$(2)$ The operator $T_{\vec{a},\ \vec{b},\ \vec{c}}$ is bounded from $L_{\vec{\alpha}}^{\vec{p}}\left( T_B\times T_B \right) $ to $L^{\infty}\left( T_B\times T_B \right)$.
	\\
	
	$(3)$ The parameters satisfy that, for any $i\in \{1,2\}$,
	$$
	\begin{cases}
	a_i>0,\ \alpha _i+1<p_i(b_i+1),\\
	c_i=n+1+a_i+b_i+\lambda_i,\\
	\end{cases}
	$$
	where
	$$
	\lambda _i=-\frac{n+1+\alpha _i}{p_i}.
	$$
\end{theorem}

\begin{proof}
	$(1)\Rightarrow(2)$ is trivial; $(2)\Rightarrow(3)$ comes from Lemma \ref{lem b5}; $(3)\Rightarrow(1)$ is derived from Lemma \ref{lem 1 5}.	
\end{proof}

\begin{theorem}\label{th 6}
	Let $\vec{p}:=\left( 1,1 \right)$ and $\vec{q}:=\left( \infty,\infty \right) $. Then the following conclusions are equivalent.
	\\
	
	$(1)$ The operator $S_{\vec{a},\ \vec{b},\ \vec{c}}$ is bounded from $L_{\vec{\alpha}}^{\vec{1}}\left( T_B\times T_B \right) $ to $L^{\infty}\left( T_B\times T_B \right)$.
	\\
	
	$(2)$ The operator $T_{\vec{a},\ \vec{b},\ \vec{c}}$ is bounded from $L_{\vec{\alpha}}^{\vec{1}}\left( T_B\times T_B \right) $ to $L^{\infty}\left( T_B\times T_B \right)$.
	\\
	
	$(3)$ The parameters satisfy that, for any $i\in \{1,2\}$,
	$$
	\begin{cases}
	a_i\ge0,\ b_i\ge\alpha _i,\\
	c_i=n+1+a_i+b_i+\lambda_i,\\
	\end{cases}
	$$
	where
	$$
	\lambda _i=-n-1-\alpha _i.
	$$
\end{theorem}

\begin{proof}
	$(1)\Rightarrow(2)$ is trivial; $(2)\Rightarrow(3)$ comes from Lemma \ref{lem b6}; $(3)\Rightarrow(1)$ is derived from Lemma \ref{lem 1 6}.	
\end{proof}

\begin{theorem}\label{th 7}
	Let $\vec{p}:=\left( 1,p_2 \right)$ and $\vec{q}:=\left( \infty,\infty \right) $ satisfy $1<p_2 <\infty$. Then the following conclusions are equivalent.
	\\
	
	$(1)$ The operator $S_{\vec{a},\ \vec{b},\ \vec{c}}$ is bounded from $L_{\vec{\alpha}}^{\vec{p}}\left( T_B\times T_B \right) $ to $L^{\infty}\left( T_B\times T_B \right)$.
	\\
	
	$(2)$ The operator $T_{\vec{a},\ \vec{b},\ \vec{c}}$ is bounded from $L_{\vec{\alpha}}^{\vec{p}}\left( T_B\times T_B \right) $ to $L^{\infty}\left( T_B\times T_B \right)$.
	\\
	
	$(3)$ The parameters satisfy that, for any $i\in\{1,2\}$,
	$$
	\begin{cases}
	a_1\ge0,\ a_2>0,\\
	b_1\ge \alpha_1,\ p_2\left( b_2+1 \right) >\alpha _2+1,\\
	c_i=n+1+a_i+b_i+\lambda_i,\\
	\end{cases}
	$$
	where
	$$
	\begin{cases}
	\lambda _1=-n-1-\alpha _1,\\
	\lambda_2=-\frac{n+1+\alpha _2}{p_2}.\\
	\end{cases}
	$$	
\end{theorem}

\begin{proof}
	$(1)\Rightarrow(2)$ is trivial; $(2)\Rightarrow(3)$ comes from Lemma \ref{lem b7}; $(3)\Rightarrow(1)$ is derived from Lemma \ref{lem 1 7}.	
\end{proof}

\begin{theorem}\label{th 8}
	Let $\vec{p}:=\left( p_1,1 \right)$ and $\vec{q}:=\left( \infty,\infty \right) $ satisfy $1<p_1 <\infty$. Then the following conclusions are equivalent.
	\\
	
	$(1)$ The operator $S_{\vec{a},\ \vec{b},\ \vec{c}}$ is bounded from $L_{\vec{\alpha}}^{\vec{p}}\left( T_B\times T_B \right) $ to $L^{\infty}\left( T_B\times T_B \right)$.
	\\
	
	$(2)$ The operator $T_{\vec{a},\ \vec{b},\ \vec{c}}$ is bounded from $L_{\vec{\alpha}}^{\vec{p}}\left( T_B\times T_B \right) $ to $L^{\infty}\left( T_B\times T_B \right)$.
	\\
	
	$(3)$ The parameters satisfy that, for any $i\in\{1,2\}$,
	$$
	\begin{cases}
	a_1>0,\ a_2\ge0, \\
	p_1\left( b_1+1 \right) >\alpha _1+1,\ b_2\ge \alpha_2,\\
	c_i=n+1+a_i+b_i+\lambda_i,\\
	\end{cases}
	$$
	where
	$$
	\begin{cases}
	\lambda_1=-\frac{n+1+\alpha _1}{p_1},\\
	\lambda _2=-n-1-\alpha _2.\\
	\end{cases}
	$$
\end{theorem}

\begin{proof}
	$(1)\Rightarrow(2)$ is trivial; $(2)\Rightarrow(3)$ comes from Lemma \ref{lem b8}; $(3)\Rightarrow(1)$ is derived from Lemma \ref{lem 1 8}.	
\end{proof}

\begin{theorem}\label{th 9}
	Let $\vec{p}:=\left( \infty,\infty \right)$ and $\vec{q}:=\left( \infty,\infty \right) $. Then the following conclusions are equivalent.
	\\
	
	$(1)$ The operator $S_{\vec{a},\ \vec{b},\ \vec{c}}$ is bounded from $L^{\infty}\left( T_B\times T_B \right) $ to $L^{\infty}\left( T_B\times T_B \right)$.
	\\
	
	$(2)$ The operator $T_{\vec{a},\ \vec{b},\ \vec{c}}$ is bounded from $L^{\infty}\left( T_B\times T_B \right) $ to $L^{\infty}\left( T_B\times T_B \right)$.
	\\
	
	$(3)$ The parameters satisfy that, for any $i\in \{1,2\}$,
	$$
	\begin{cases}
	a_i>0,\ b_i>-1,\\
	c_i=n+1+a_i+b_i.\\
	\end{cases}
	$$
\end{theorem}

\begin{proof}
	$(1)\Rightarrow(2)$ is trivial; $(2)\Rightarrow(3)$ comes from Lemma \ref{lem b9}; $(3)\Rightarrow(1)$ is derived from Lemma \ref{lem 1 9}.	
\end{proof}

\begin{theorem}\label{th 10}
	Let $\vec{p}:=\left( 1,\infty \right)$ and $\vec{q}:=\left( \infty,\infty \right) $. Then the following conclusions are equivalent.
	\\
	
	$(1)$ The operator $S_{\vec{a},\ \vec{b},\ \vec{c}}$ is bounded from $L_{\vec{\alpha}}^{\vec{p}}\left( T_B\times T_B \right) $ to $L^{\infty}\left( T_B\times T_B \right)$.
	\\
	
	$(2)$ The operator $T_{\vec{a},\ \vec{b},\ \vec{c}}$ is bounded from $L_{\vec{\alpha}}^{\vec{p}}\left( T_B\times T_B \right) $ to $L^{\infty}\left( T_B\times T_B \right)$.
	\\
	
	$(3)$ The parameters satisfy that, for any $i\in \{1,2\}$,
	$$
	\begin{cases}
	a_1\ge0,\ a_2>0, \\
	b_1 \ge \alpha _1,\ b_2>-1,\\
	c_i=n+1+a_i+b_i+\lambda_i,\\
	\end{cases}
	$$
	where
	$$
	\begin{cases}
	\lambda _1=-n-1-\alpha _1,\\
	\lambda_2=0.\\
	\end{cases}
	$$
\end{theorem}

\begin{proof}
	$(1)\Rightarrow(2)$ is trivial; $(2)\Rightarrow(3)$ comes from Lemma \ref{lem b10}; $(3)\Rightarrow(1)$ is derived from Lemma \ref{lem 1 10}.	
\end{proof}

\begin{theorem}\label{th 11}
	Let $\vec{p}:=\left( \infty,1 \right)$ and $\vec{q}:=\left( \infty,\infty \right) $. Then the following conclusions are equivalent.
	\\
	
	$(1)$ The operator $S_{\vec{a},\ \vec{b},\ \vec{c}}$ is bounded from $L_{\vec{\alpha}}^{\vec{p}}\left( T_B\times T_B \right) $ to $L^{\infty}\left( T_B\times T_B \right)$.
	\\
	
	$(2)$ The operator $T_{\vec{a},\ \vec{b},\ \vec{c}}$ is bounded from $L_{\vec{\alpha}}^{\vec{p}}\left( T_B\times T_B \right) $ to $L^{\infty}\left( T_B\times T_B \right)$.
	\\
	
	$(3)$ The parameters satisfy that, for any $i\in \{1,2\}$,
	$$
	\begin{cases}
	a_1>0,\ a_2\ge 0,\\
	b_1>-1,\  b_2\ge \alpha_2,\\
	c_i=n+1+a_i+b_i+\lambda_i,\\
	\end{cases}
	$$
	where
	$$
	\begin{cases}
	\lambda_1=0,\\
	\lambda _2=-n-1-\alpha _2.\\
	\end{cases}
	$$
\end{theorem}

\begin{proof}
	$(1)\Rightarrow(2)$ is trivial; $(2)\Rightarrow(3)$ comes from Lemma \ref{lem b11}; $(3)\Rightarrow(1)$ is derived from Lemma \ref{lem 1 11}.	
\end{proof}

\begin{theorem}\label{th 12}
	Let $\vec{p}:=\left( \infty,p_2 \right)$ and $\vec{q}:=\left( \infty,\infty \right) $ satisfy $1<p_2 <\infty$. Then the following conclusions are equivalent.
	\\
	
	$(1)$ The operator $S_{\vec{a},\ \vec{b},\ \vec{c}}$ is bounded from $L_{\vec{\alpha}}^{\vec{p}}\left( T_B\times T_B \right)$ to $L^{\infty}\left( T_B\times T_B \right)$.
	\\
	
	$(2)$ The operator $T_{\vec{a},\ \vec{b},\ \vec{c}}$ is bounded from $L_{\vec{\alpha}}^{\vec{p}}\left( T_B\times T_B \right)$ to $L^{\infty}\left( T_B\times T_B \right)$.
	\\
	
	$(3)$ The parameters satisfy that, for any $i\in \{1,2\}$,
	$$
	\begin{cases}
	a_1>0,\ a_2>0,\\
	b_1>-1,\ p_2\left( b_2+1 \right) >\alpha _2+1,\\
	c_i=n+1+a_i+b_i+\lambda _i,\\
	\end{cases}	
	$$
	where
	$$
	\begin{cases}
	\lambda _1=0,\\
	\lambda_2=-\frac{n+1+\alpha _2}{p_2}.\\
	\end{cases}
	$$
\end{theorem}

\begin{proof}
	$(1)\Rightarrow(2)$ is trivial; $(2)\Rightarrow(3)$ comes from Lemma \ref{lem b12}; $(3)\Rightarrow(1)$ is derived from Lemma \ref{lem 1 12}.	
\end{proof}

\begin{theorem}\label{th 13}
	Let $\vec{p}:=\left(p_1,\infty \right)$ and $\vec{q}:=\left( \infty,\infty \right) $ satisfy $1<p_1 <\infty$. Then the following conclusions are equivalent.
	\\
	
	$(1)$ The operator $S_{\vec{a},\ \vec{b},\ \vec{c}}$ is bounded from $L_{\vec{\alpha}}^{\vec{p}}\left( T_B\times T_B \right)$ to $L^{\infty}\left( T_B\times T_B \right)$.
	\\
	
	$(2)$ The operator $T_{\vec{a},\ \vec{b},\ \vec{c}}$ is bounded from $L_{\vec{\alpha}}^{\vec{p}}\left( T_B\times T_B \right)$ to $L^{\infty}\left( T_B\times T_B \right)$.
	\\
	
	$(3)$ The parameters satisfy that, for any $i\in \{1,2\}$,
	$$
	\begin{cases}
	a_1>0,\ a_2>0,\\
	p_1\left( b_1+1 \right) >\alpha _1+1,\ b_2>-1,\\
	c_i=n+1+a_i+b_i+\lambda_i,\\
	\end{cases}	
	$$
	where
	$$
	\begin{cases}
	\lambda_1=-\frac{n+1+\alpha _1}{p_1},\\
	\lambda _2=0.\\
	\end{cases}
	$$
\end{theorem}

\begin{proof}
	$(1)\Rightarrow(2)$ is trivial; $(2)\Rightarrow(3)$ comes from Lemma \ref{lem b13}; $(3)\Rightarrow(1)$ is derived from Lemma \ref{lem 1 13}.	
\end{proof}

\section{Applications}	
\ \ \ \ 
In this section, we will give three applications of the main theorems in this paper, that is, to study the boundedness of three kinds of special integral operators.

We study the $L_{\vec{\gamma}}^{\vec{p}}\left( T_B\times T_B \right)$-$L_{\vec{\beta}}^{\vec{q}}\left( T_B\times T_B \right)$ boundedness of the following operators
$$
T_{\vec{c}}^{\vec{\gamma}}f\left( z,w \right) =\int_{T_B}{\int_{T_B}{\frac{f\left( u,\eta \right)}{\rho \left( z,u \right) ^{c_1}\rho \left( w,\eta \right) ^{c_2}}dV_{\gamma _1}\left( u \right)}dV_{\gamma _2}\left( \eta \right)},
$$
where $c_1,\ c_2>0$ and $\gamma_1,\ \gamma_2>-1$.

Notice that $T_{\vec{c}}^{\vec{\gamma}}=T_{\vec{0},\ \vec{\gamma},\ \vec{c}}$, we have the following results.  
\\

Corollaries \ref{T you jie1} and \ref{T you jie2} state when operator $T_{\vec{c}}^{\vec{\gamma}}$ is bounded.

\begin{corollary}\label{T you jie1}
	If $\vec{p}:=\left( p_1,p_2 \right)$ and $\vec{q}:=\left( q_1,q_2 \right) $ satisfy $1<p_-\le p_+\le q_-\le q_+ <\infty$, then the operator $T_{\vec{c}}^{\vec{\gamma}}$ is bounded from $L_{\vec{\gamma}}^{\vec{p}}\left( T_B\times T_B \right) $ to $L_{\vec{\beta}}^{\vec{q}}\left( T_B\times T_B \right) $ if and only if the parameters satisfy that, for any $i\in \{1,2\}$,
	$$
	\begin{cases}
	\beta _i>-1,\\
	c_i=n+1+\gamma_i+\lambda_i,\\
	\end{cases}
	$$
	\\
	where 
	$$
	\lambda_i =\frac{n+1+\beta _i}{q_i}-\frac{n+1+\gamma _i}{p_i}.
	$$
\end{corollary}

\begin{corollary}\label{T you jie2}
	If $\vec{p}:=\left( 1,1 \right)$ and $\vec{q}:=\left( \infty,\infty \right) $, then the operator $T_{\vec{c}}^{\vec{\gamma}}$ is bounded from $L_{\vec{\gamma}}^{\vec{1}}\left( T_B\times T_B \right) $ to $L^{\infty}\left( T_B\times T_B \right) $ if and only if
	$c_1=c_2=0$.
\end{corollary}

Corollary \ref{T wu jie} states when operator $T_{\vec{c}}^{\vec{\gamma}}$ is unbounded.

\begin{corollary}\label{T wu jie}
	If $\vec{p}:=\left( p_1,p_2 \right)$ and $\vec{q}:=\left( q_1,q_2 \right) $ satisfy one of the following conditions : 
	
	$(1)$ $p_1=p_2=1$ and $q_1,q_2\ge 1$;
\\	

	$(2)$ $1=p_-< p_+<  q_- \le q_+<\infty$;
\\	

	$(3)$ $1\le p_-\le p_+\le\infty$,  $\vec{p}\ne \left( 1,1 \right)$  and $\left( q_1,q_2 \right)=\left( \infty,\infty \right)$;
\\	

	then the operator $T_{\vec{c}}^{\vec{\gamma}}$ is unbounded on $L_{\vec{\gamma}}^{\vec{p}}\left( T_B\times T_B \right)$.
\end{corollary}

Next, we study the boundedness of Bergman-type projections.

Let $\vec{\gamma}:=\left(\gamma_1,\gamma_2\right)\in \left(-1,\infty \right)^2$, the weighted multiparameter Bergman-type projection is as follows :
$$
P_{\vec{\gamma}}f(z,w):=\int_{T_B}{\int_{T_B}{\frac{f\left( u,\eta \right)}{\rho \left( z,u \right) ^{n+1+\gamma _1}\rho \left( w,\eta \right) ^{n+1+\gamma _2}}dV_{\gamma _1}\left( u \right)}dV_{\gamma _2}\left( \eta \right)}.
$$

Notice that $P_{\vec{\gamma}}=T_{\overrightarrow{n+1+\gamma }}^{\vec{\gamma}}=T_{\vec{0},\ \vec{\gamma},\ \overrightarrow{n+1+\gamma }}$, hence we have the following results. 
\\

The following corollaries \ref{P you jie 1} to \ref{P you jie 4} give the case when the operator $P_{\vec{\gamma}}$ is bounded from $L_{\vec{\alpha}}^{\vec{p}}\left( T_B\times T_B \right) $ to $L_{\vec{\beta}}^{\vec{q}}\left( T_B\times T_B \right) $. Corollary \ref{P wu jie} gives when the operator $P_{\vec{\gamma}}$ is unbounded on $L_{\vec{\alpha}}^{\vec{p}}\left( T_B\times T_B \right) $.

\begin{corollary}\label{P you jie 1}
	If $\vec{p}:=\left( p_1,p_2 \right)$ and $\vec{q}:=\left( q_1,q_2 \right) $ satisfy $1<p_-\le p_+\le q_-\le q_+ <\infty$, then the operator $P_{\vec{\gamma}}$ is bounded from $L_{\vec{\alpha}}^{\vec{p}}\left( T_B\times T_B \right) $ to $L_{\vec{\beta}}^{\vec{q}}\left( T_B\times T_B \right) $ if and only if the parameters satisfy that, for any $i\in \{1,2\}$,
	$$
	\begin{cases}
	p_i\left(\gamma_i+1\right)>\alpha_i+1,\\
	p_i\left( n+1+\beta _i \right)=q_i\left( n+1+\alpha _i \right).\\
	\end{cases}
	$$
\end{corollary}

\begin{corollary}\label{P you jie 2}
	If $\vec{p}:=\left(1,1 \right)$ and $\vec{q}:=\left( q_1,q_2 \right) $ satisfy $1\le q_-\le q_+ <\infty$, then the operator $P_{\vec{\gamma}}$ is bounded from $L_{\vec{\alpha}}^{\vec{1}}\left( T_B\times T_B \right) $ to $L_{\vec{\beta}}^{\vec{q}}\left( T_B\times T_B \right) $ if and only if the parameters satisfy that, for any $i\in \{1,2\}$,
	$$
	\begin{cases}
	\gamma_i>\alpha_i,\\
	 n+1+\beta _i=q_i\left( n+1+\alpha _i \right).\\
	\end{cases}
	$$
\end{corollary}

\begin{corollary}\label{P you jie 3}
	If $\vec{p}:=\left(p_1,1 \right)$ and $\vec{q}:=\left( q_1,q_2 \right) $ satisfy $1<p_1\le  q_- \le q_+<\infty$, then the operator $P_{\vec{\gamma}}$ is bounded from $L_{\vec{\alpha}}^{\vec{p}}\left( T_B\times T_B \right) $ to $L_{\vec{\beta}}^{\vec{q}}\left( T_B\times T_B \right) $ if and only if the parameters satisfy that,
	$$
	\begin{cases}	
	p_1\left( \gamma_1+1 \right)>\alpha _1+1,\ \gamma_2>\alpha_2,\\
	p_1\left(n+1+\beta _1\right)=q_1\left( n+1+\alpha _1\right),\\
	n+1+\beta _2=q_2\left( n+1+\alpha _2 \right).\\
	\end{cases}
	$$
\end{corollary}

\begin{corollary}\label{P you jie 4}
	If $\vec{p}:=\left(1,p_2 \right)$ and $\vec{q}:=\left( q_1,q_2 \right) $ satisfy $1<p_2\le  q_- \le q_+<\infty$, then the operator $P_{\vec{\gamma}}$ is bounded from $L_{\vec{\alpha}}^{\vec{p}}\left( T_B\times T_B \right) $ to $L_{\vec{\beta}}^{\vec{q}}\left( T_B\times T_B \right) $ if and only if the parameters satisfy that, 
	$$
	\begin{cases}	
	\gamma_1>\alpha_1,\ p_2\left( \gamma_2+1 \right)>\alpha _2+1,\\
	n+1+\beta _1=q_1\left( n+1+\alpha _1 \right),\\
	p_2\left(n+1+\beta _2\right)=q_2\left( n+1+\alpha _2\right).\\
	\end{cases}
	$$
\end{corollary}

\begin{corollary}
	If $\vec{p}:=\left( 1,1 \right)$ and $\vec{q}:=\left( \infty,\infty \right) $, then the operator $P_{\vec{\gamma}}$ is bounded from $L_{\vec{\alpha}}^{\vec{1}}\left( T_B\times T_B \right) $ to $L^{\infty}\left( T_B\times T_B \right) $ if and only if the parameters satisfy that, for any $i\in\{1,2\}$,
	$$
	\begin{cases}	
	\gamma_i\ge\alpha_i,\\
	\alpha_i=-\left( n+1\right).\\
	\end{cases}
	$$
\end{corollary}

\begin{corollary}\label{P wu jie}
	If $\vec{p}:=\left( p_1,p_2 \right)$ and $\vec{q}:=\left( q_1,q_2 \right) $ satisfy the following conditions : 
	$$
	1\le p_-\le p_+\le \infty,\ \vec{p}\ne \left( 1,1 \right)\ \text{and}\ \left( q_1,q_2 \right) =\left( \infty ,\infty \right) ,
	$$	
	then the operator $P_{\vec{\gamma}}$ is unbounded on $L_{\vec{\alpha}}^{\vec{p}}\left( T_B\times T_B \right)$.
\end{corollary}

At the end of this section, we study the boundedness of the weighted multiparameter Berezin-type transform $B_{\vec{\gamma}}$ which is given by
$$
B_{\vec{\gamma}}f(z,w):=\int_{T_B}{\int_{T_B}{\frac{\rho \left( z \right) ^{n+1+\gamma _1}\rho \left( w \right) ^{n+1+\gamma _2}}{\left| \rho \left( z,u \right) \right|^{2\left(n+1+\gamma _1\right)}\left| \rho \left( w,\eta \right) \right|^{2\left(n+1+\gamma _2\right)}}f\left( u,\eta \right) dV_{\gamma _1}\left( u \right)}dV_{\gamma _2}\left( \eta \right)},
$$
where $\vec{\gamma}:=\left(\gamma_1,\gamma_2\right)\in \left(-1,\infty \right)\times\left(-1,\infty \right)$.

Notice that $B_{\vec{\gamma}}=S_{\overrightarrow{n+1+\gamma }, \overrightarrow{\gamma}, 2\overrightarrow{\left( n+1+\gamma \right)}}$, therefore we have the following results. 
\\

The following corollaries \ref{B you jie 1} to \ref{B you jie 11} give the case when the operator $B_{\vec{\gamma}}$ is bounded from $L_{\vec{\alpha}}^{\vec{p}}\left( T_B\times T_B \right) $ to $L_{\vec{\beta}}^{\vec{q}}\left( T_B\times T_B \right) $.

\begin{corollary}\label{B you jie 1}
	If $\vec{p}:=\left( p_1,p_2 \right)$ and $\vec{q}:=\left( q_1,q_2 \right) $ satisfy $1<p_-\le p_+\le q_-\le q_+ <\infty$, then the operator $B_{\vec{\gamma}}$ is bounded from $L_{\vec{\alpha}}^{\vec{p}}\left( T_B\times T_B \right) $ to $L_{\vec{\beta}}^{\vec{q}}\left( T_B\times T_B \right) $ if and only if the parameters satisfy that, for any $i\in \{1,2\}$,
	$$
	\begin{cases}
	-q_i\left( n+1+\gamma_i \right)<\beta _i+1,\  \alpha _i+1<p_i\left( \gamma_i+1 \right) ,\\
	p_i\left( n+1+\beta _i \right)=q_i\left( n+1+\alpha _i \right).\\
	\end{cases}
	$$
\end{corollary}

\begin{corollary}\label{B you jie 2}
	If $\vec{p}:=\left(1,1 \right)$ and $\vec{q}:=\left( q_1,q_2 \right) \in [1,\infty)\times[1,\infty)$, then the operator $B_{\vec{\gamma}}$ is bounded from $L_{\vec{\alpha}}^{\vec{1}}\left( T_B\times T_B \right) $ to $L_{\vec{\beta}}^{\vec{q}}\left( T_B\times T_B \right) $ if and only if the parameters satisfy that, for any $i\in \{1,2\}$,
	$$
	\begin{cases}
	-q_i\left( n+1+\gamma_i \right)<\beta _i+1,\  \alpha _i< \gamma_i ,\\
	n+1+\beta _i=q_i\left( n+1+\alpha _i \right).\\
	\end{cases}
	$$
\end{corollary}

\begin{corollary}\label{B you jie 3}
	If $\vec{p}:=\left(p_1,1 \right)$ and $\vec{q}:=\left( q_1,q_2 \right) $ satisfy $1<p_1\le  q_- \le q_+<\infty$, then the operator $B_{\vec{\gamma}}$ is bounded from $L_{\vec{\alpha}}^{\vec{p}}\left( T_B\times T_B \right) $ to $L_{\vec{\beta}}^{\vec{q}}\left( T_B\times T_B \right) $ if and only if the parameters satisfy that,
	$$
	\begin{cases}	
	p_1\left( \gamma_1+1 \right)>\alpha _1+1,\ \gamma_2>\alpha_2,\\
	p_1\left(n+1+\beta _1\right)=q_1\left( n+1+\alpha _1\right),\\
	n+1+\beta _2=q_2\left( n+1+\alpha _2 \right).\\
	\end{cases}
	$$
\end{corollary}

\begin{corollary}\label{B you jie 4}
	If $\vec{p}:=\left(1,p_2 \right)$ and $\vec{q}:=\left( q_1,q_2 \right) $ satisfy $1<p_2\le  q_- \le q_+<\infty$, then the operator $B_{\vec{\gamma}}$ is bounded from $L_{\vec{\alpha}}^{\vec{p}}\left( T_B\times T_B \right) $ to $L_{\vec{\beta}}^{\vec{q}}\left( T_B\times T_B \right) $ if and only if the parameters satisfy that, 
	$$
	\begin{cases}	
	\gamma_1>\alpha_1,\ p_2\left( \gamma_2+1 \right)>\alpha _2+1,\\
	n+1+\beta _1=q_1\left( n+1+\alpha _1 \right),\\
	p_2\left(n+1+\beta _2\right)=q_2\left( n+1+\alpha _2\right).\\
	\end{cases}
	$$
\end{corollary}

\begin{corollary}\label{B you jie 5}
	If $\vec{p}:=\left(p_1,p_2 \right)$ and $\vec{q}:=\left(\infty,\infty\right) $ satisfy $1<p_-\le p_+ <\infty$, then the operator $B_{\vec{\gamma}}$ is bounded from $L_{\vec{\alpha}}^{\vec{p}}\left( T_B\times T_B \right) $ to $L^{\infty}\left( T_B\times T_B \right) $ if and only if the parameters satisfy that, for any $i\in\{1,2\}$,
	$$
	\gamma_i>\alpha _i=-n-1.
	$$
\end{corollary}

\begin{corollary}\label{B you jie 6}
	If $\vec{p}:=\left( 1,1 \right)$ and $\vec{q}:=\left( \infty,\infty \right) $, then the operator $B_{\vec{\gamma}}$ is bounded from $L_{\vec{\alpha}}^{\vec{1}}\left( T_B\times T_B \right) $ to $L^{\infty}\left( T_B\times T_B \right)$ if and only if the parameters satisfy that, for any $i\in\{1,2\}$,
	$$
	\gamma_i\ge\alpha _i=-n-1.
	$$
\end{corollary}

\begin{corollary}\label{B you jie 7}
	If $\vec{p}:=\left( 1,p_2 \right)$ and $\vec{q}:=\left( \infty,\infty \right)$ satisfy $1<p_2<\infty$, then the operator $B_{\vec{\gamma}}$ is bounded from $L_{\vec{\alpha}}^{\vec{p}}\left( T_B\times T_B \right) $ to $L^{\infty}\left( T_B\times T_B \right)$ if and only if the parameters satisfy that, 
	$$
	\begin{cases}
	\gamma_1\ge \alpha _1=-n-1,\\
	\gamma_2>\alpha _2=-n-1.\\
	\end{cases}
	$$
\end{corollary}

\begin{corollary}\label{B you jie 8}
	If $\vec{p}:=\left( p_1,1 \right)$ and $\vec{q}:=\left( \infty,\infty \right)$ satisfy $1<p_1<\infty$, then the operator $B_{\vec{\gamma}}$ is bounded from $L_{\vec{\alpha}}^{\vec{p}}\left( T_B\times T_B \right) $ to $L^{\infty}\left( T_B\times T_B \right)$ if and only if the parameters satisfy that, 
	$$
	\begin{cases}
	\gamma_1> \alpha _1=-n-1,\\
	\gamma_2\ge \alpha _2=-n-1.\\
	\end{cases}
	$$
\end{corollary}

\begin{corollary}\label{B you jie 9}
	If $\vec{p}:=\left( \infty,\infty \right)$ and $\vec{q}:=\left( \infty,\infty \right) $, then the operator $B_{\vec{\gamma}}$ is always bounded from $L^{\infty}\left( T_B\times T_B \right) $ to $L^{\infty}\left( T_B\times T_B \right)$.
\end{corollary}

\begin{corollary}\label{B you jie 10}
	If $\vec{p}:=\left( p_1,p_2 \right)$ and $\vec{q}:=\left( \infty,\infty \right) $ satisfy $1=p_-<p_+=\infty$, then the operator $B_{\vec{\gamma}}$ is bounded from $L_{\vec{\alpha}}^{\vec{p}}\left( T_B\times T_B \right) $ to $L^{\infty}\left( T_B\times T_B \right)$ if and only if the parameters satisfy that, for any $i\in\{1,2\}$,
	$\gamma_i\ge \alpha_i=-n-1.$
\end{corollary}

\begin{corollary}\label{B you jie 11}
	If $\vec{p}:=\left( p_1, p_2\right)$ and $\vec{q}:=\left( \infty,\infty \right) $ satisfy $1<p_-<p_+=\infty$, then the operator $B_{\vec{\gamma}}$ is bounded from $L_{\vec{\alpha}}^{\vec{p}}\left( T_B\times T_B \right) $ to $L^{\infty}\left( T_B\times T_B \right)$ if and only if the parameters satisfy that, for any $i\in\{1,2\}$,
	$\gamma_i>\alpha_i=-n-1.$
\end{corollary}

\end{document}